\documentclass[12pt,a4paper]{article}

\usepackage[utf8]{inputenc}
\usepackage[T1]{fontenc}
\usepackage[english]{babel}

\usepackage{amsmath, amssymb, amsfonts, amsthm}
\usepackage{mathtools}
\usepackage{graphicx}
\usepackage{xcolor}
\usepackage{enumitem}
\usepackage{hyperref}

\usepackage{pgfplots}
\pgfplotsset{compat=1.17}
\usepackage{booktabs}

\usepackage[margin=1in]{geometry}

\usepackage{float}
\usepackage{tabularx}
\usepackage{array}
\newcolumntype{C}{>{\centering\arraybackslash}X}

\hypersetup{
    colorlinks=true,
    linkcolor=blue!50!black,
    citecolor=blue!50!black,
    urlcolor=blue!50!black
}

\newcommand{\added}[1]{#1}

\newcommand{\hl}[1]{#1}
\newcommand{\highlighting}[1]{#1}
\newcommand{\PublishersNote}[1]{}

\newcommand{\bE}{\mathbb E}
\newcommand{\bP}{\mathbb P}
\newcommand{\bQ}{\mathbb Q}
\newcommand{\bR}{\mathbb R}
\newcommand{\R}{\mathbb R}
\newcommand{\rd}{\mathrm d}
\newcommand{\rw}{\mathrm w}
\newcommand{\rT}{\mathrm T}
\newcommand{\cN}{\mathcal N}

\theoremstyle{plain}
\newtheorem{theorem}{Theorem}[section]
\newtheorem{lemma}[theorem]{Lemma}
\newtheorem{proposition}[theorem]{Proposition}
\newtheorem{corollary}[theorem]{Corollary}

\theoremstyle{definition}
\newtheorem{definition}[theorem]{Definition}
\newtheorem{assumption}[theorem]{Assumption}

\theoremstyle{remark}
\newtheorem{remark}[theorem]{Remark}
\newtheorem{example}[theorem]{Example}

\newenvironment{Theorem}{\begin{theorem}}{\end{theorem}}
\newenvironment{Lemma}{\begin{lemma}}{\end{lemma}}
\newenvironment{Proposition}{\begin{proposition}}{\end{proposition}}
\newenvironment{Corollary}{\begin{corollary}}{\end{corollary}}

\title{Weighted Chernoff Information and Optimal Loss Exponent\\[2pt]
in Context-Sensitive Hypothesis Testing}

\author{
Mark Kelbert\thanks{Laboratory of Stochastic Analysis and its Applications, Department of Statistics and Data Analysis, National Research University Higher School of Economics, 101000 Moscow, Russia, and Department of Mathematics, Swansea University, Swansea SA2 8PP, UK. Email: \texttt{mkelbert@hse.ru}.}
\and
El'mira Yu.~Kalimulina\thanks{Institute for Information Transmission Problems, Russian Academy of Sciences (IITP RAS), 127051 Moscow, Russia, and Faculty of Mechanics and Mathematics, Lomonosov Moscow State University, 119991 Moscow, Russia. Email: \texttt{eyk@iitp.ru}. ORCID: 0000-0001-7158-040X.}
}

\date{\today}

\begin{document}

\maketitle

\begin{abstract}
We study binary hypothesis testing for i.i.d.\ observations under a multiplicative context weight. For the optimal weighted total loss, defined as the sum of weighted type-I and type-II losses, we prove the logarithmic asymptotic
\[
L_n^*=\exp\{-nD_C^{\rw}(\bP,\bQ)+o(n)\},
\qquad n\to\infty,
\]
where $D_C^{\rw}$ is the weighted Chernoff information. The single-letter form of the exponent relies on a structural assumption that the weight factorises across observations, $\varphi(x_1^n)=\prod_{i=1}^n\varphi(x_i)$; this restriction is essential for the single-letter representation and should be distinguished from the weaker qualitative description ``multiplicative context weight''. The proof embeds the weighted geometric mixtures $\varphi p^\alpha q^{1-\alpha}$ into a likelihood-ratio exponential family and identifies the rate through its log-normaliser. We also derive concentration bounds for the tilted weighted log-likelihood, obtain closed forms for Gaussian, Poisson, and exponential models, and extend the exponent characterisation to finitely many hypotheses.



\medskip
\noindent\textbf{Published version:} \emph{Entropy} \textbf{2026}, \emph{28}(5), 536. DOI: \href{https://doi.org/10.3390/e28050536}{10.3390/e28050536}.
\end{abstract}

\tableofcontents

\bigskip

\section{Introduction}\label{sec:intro}
 Let ${\cal X}$ be a Polish space with its Borel $\sigma$-algebra and let
$X_1^n=(X_1,\ldots,X_n)$ be i.i.d.\ ${\cal X}$-valued observations. We consider the simple hypotheses
\[
H_0:\ X_1^n\sim \bP^{\otimes n}
\qquad\text{versus}\qquad
H_1:\ X_1^n\sim \bQ^{\otimes n},
\]
where $\bP$ and $\bQ$ are probability measures on ${\cal X}$ dominated by a reference measure $\mu$.
Without loss of generality, one may take $\mu=\frac12(\bP+\bQ)$ and write
$p=\frac{\rd\bP}{\rd\mu}$ and $q=\frac{\rd\bQ}{\rd\mu}$.
In the unweighted setting, the optimal sum of type-I and type-II error probabilities
is characterized by $\mathrm{TV}(\bP^{\otimes n},\bQ^{\otimes n})$ and can be written as 
\begin{equation}\label{eq:OptimalSumWithoutWeight}
\int\limits_{{\mathcal X}^n}\min\{p(x_1^n),q(x_1^n)\}\,\rd\mu^{\otimes n}(x_1^n),
~~
p(x_1^n)=\prod\limits_{i=1}^n p(x_i), ~~ q(x_1^n)=\prod\limits_{i=1}^n q(x_i).
\end{equation}

In the standard (unweighted) Bayesian setting,
the decay rate of the optimal total error probability is governed by the Chernoff information \hl{}  \cite{Ch,Hoeffding65}:
\begin{align}\label{eq:ChernoffDistDefinition}
\rho_{\alpha}(p,q)&:=\int\limits_{\cal X} p(x)^{\alpha}q(x)^{1-\alpha}\,\rd\mu(x),\qquad \alpha\in[0,1],
\nonumber\\
\rho(p,q)&:=\inf_{\alpha\in[0,1]}\rho_{\alpha}(p,q),
\nonumber\\
D_C(\bP,\bQ)&:=-\ln \rho(p,q)=\max_{\alpha\in[0,1]}\bigl[-\ln \rho_{\alpha}(p,q)\bigr].
\end{align}

Here $\rho_{\alpha}$ is usually called the 
    $\alpha$-skewed Bhattacharyya affinity coefficient, and
    $\rho(p,q)=\inf\limits_{\alpha\in [0,1]}\rho_{\alpha}(p,q)$ is the affinity coefficient.
    In view of H\"older's inequality, $\rho_{\alpha}(p,q)\in [0,1]$.
  
    Chernoff also introduced an asymptotic efficiency notion for comparing two experimental designs  $e=\frac{\ln\rho_1}{\ln\rho_2}$ such that $n$ observations on one test are equivalent (i.e., they give asymptotically the same total loss as $n\to\infty$) to $en$ observations on another test; see \cite{Ch}.
    
\added{The paper studies a context-sensitive (weighted) analogue of this criterion and the logarithmic asymptotics of the optimal total loss as $n\to\infty$, in the framework of \cite{KS1,KS2}. In the weighted setting, a nonnegative weight function $\varphi(x_1^n)$ reweights the loss of a wrong decision according to the realised sample. Thus, $\varphi$ acts as a context factor that changes the relevance of different observations for the statistical task.}

{Weights of this form arise naturally whenever observations are not equally informative for the inference task. Two canonical mechanisms produce such $\varphi$. In \emph{importance-type reweighting}, samples drawn under a proposal density $g$ are used to perform inference with respect to a target $h$, and the Radon--Nikodym factor $\varphi(x)=h(x)/g(x)$ enters the loss as a strictly positive (non-indicator) tilt; this is the mechanism underlying the context-sensitive framework of \cite{KS1,KS2}. 
}

In applications, the informational value of an observation often depends on the underlying channel state. A canonical example, directly relevant to multiple hypothesis testing of transmission regimes, is a mobile communication channel modulated by a multi-zone coverage process (e.g., strong / weak / outage) along the receiver trajectory: samples acquired in outage carry little information about the regime and are weighted accordingly. Such reliability-weighted aggregation in multi-state channels was studied within multi-valued frameworks in  \cite{KalimulinaAICT2021,EsinKalimulina2026}.

Under the standard assumption that the modulating state at time $i$ is determined by $X_i$ alone, the resulting weight is a strictly positive bounded function $\varphi(x)$ and extends multiplicatively to $X_1^n$. The weighted Chernoff information $D_C^{\rw}(\bP,\bQ)$ then quantifies the effective discrimination rate under channel-dependent reliability and reduces to the classical rate $D_C(\bP,\bQ)$ in the limit $\varphi\equiv 1$. Further parametric instances (Gaussian, Poisson, exponential) are worked out in Section~\ref{sec:examples}.

\added{Throughout we assume that the weight is compatible with the i.i.d.\ structure and factorises across observations; by abuse of notation, $\varphi$ denotes both the one-step weight and its product extension.}

\begin{assumption}[Factorised weight]\label{ass:phi-factorised}
The weight function $\varphi({x}_1^n)$ satisfies
\begin{equation}\label{eq:phi-factorised}
\varphi({x}_1^n)=\prod_{i=1}^n\varphi(x_i), \qquad \varphi\ge 0.
\end{equation}
\end{assumption}

\added{Assumption~\ref{ass:phi-factorised} is the key single-letter hypothesis. It yields the weighted affinities
\[
\rho^{\rw}_\alpha(p,q)=\int_{\cal X}\varphi(x)p(x)^\alpha q(x)^{1-\alpha}\,\rd\mu(x),
\]
hence an additive logarithmic rate. For one observation and equal priors, the weighted Bayes risk equals
\[
\frac12\int_{\cal X}\varphi(x)\min\{p(x),q(x)\}\,\rd\mu(x).
\]
Since  $\min\{a,b\}\le a^\alpha b^{1-\alpha}$ for every $\alpha\in[0,1]$,
\[
\int_{\cal X}\varphi(x)\min\{p(x),q(x)\}\,\rd\mu(x)
\le \rho^{\rw}_\alpha(p,q),
\]
and therefore
\[
\int_{\cal X}\varphi(x)\min\{p(x),q(x)\}\,\rd\mu(x)
\le \exp\{-D_C^{\rw}(\bP,\bQ)\},
\]
where $D_C^{\rw}(\bP,\bQ):=\max_{\alpha\in[0,1]}\bigl[-\ln\rho^{\rw}_{\alpha}(p,q)\bigr]$ (see Definition~\ref{def:weighted-chernoff}). Under Assumption~\ref{ass:phi-factorised}, the same bound factorises over $n$ observations and yields the exponential scale $\exp\{-nD_C^{\rw}(\bP,\bQ)\}$. Theorem~\ref{thm:optimal-sum-loss} shows that this scale is exact on the logarithmic level.}

\subsection{\highlighting{Main} Result and Contributions }
\added{Let $L_n^*$ denote the optimal total context-sensitive loss (sum of weighted type-I and type-II losses, minimised over decision rules) for $n$ i.i.d.\ observations under Assumption~\ref{ass:phi-factorised}. Our main theorem (Theorem~\ref{thm:optimal-sum-loss}) proves the single-letter logarithmic asymptotic}
\begin{equation}\label{eq:intro-main-asymptotic}
L_n^*=\exp\{-nD_C^{\rw}(\bP,\bQ)+o(n)\},\qquad n\to\infty,
\end{equation}
\added{where the rate is the \emph{weighted Chernoff information}}
\begin{equation}\label{eq:intro-DCw}
D_C^{\rw}(\bP,\bQ)
=\max_{\alpha\in[0,1]}\left[-\ln\int_{\cal X}\varphi(x)\,p(x)^\alpha q(x)^{1-\alpha}\,\rd\mu(x)\right].
\end{equation}
\added{For $\varphi\equiv 1$, \eqref{eq:intro-DCw} reduces to the classical Chernoff information.

We also extend the exponent characterisation to a finite family of simple hypotheses: the optimal $M$-ary rate is the minimum pairwise weighted Chernoff information (cf.\ \cite{Nielsen2013} in the unweighted case). A central technical device is an exponential-family representation of the weighted geometric mixtures $\alpha\mapsto \varphi\,p^\alpha q^{1-\alpha}$. This embeds the mixtures into a likelihood-ratio exponential family and identifies the exponent through the corresponding log-normaliser. We further derive concentration bounds for tilted weighted log-likelihood ratios and closed-form expressions for $D_C^{\rw}$ in several parametric models; see Section~\ref{sec:examples}.}

{\subsection{Contributions}%
Items (N1)--(N4) below indicate new results; items (A1)--(A3) summarise definitions, geometric context, and tools adopted from the existing literature.
\begin{itemize}
\item[(N1)] \emph{(New.)} Theorem~\ref{thm:optimal-sum-loss} establishes the logarithmic asymptotic \eqref{eq:intro-main-asymptotic} for the optimal weighted total loss under the factorised weight of Assumption~\ref{ass:phi-factorised}, with rate given by the weighted Chernoff information \eqref{eq:intro-DCw}.
\item[(N2)] \emph{(New.)} The exponential-family representation of the weighted geometric mixtures $\alpha\mapsto\varphi\,p^\alpha q^{1-\alpha}$ (Section~\ref{subsec:exp-family}) and the resulting uniqueness of the optimal skewing parameter $\alpha^\ast$.
\item[(N3)] \emph{(New.)} Concentration bounds for the tilted weighted log-likelihood and the finite-$n$ tail bound of Theorem~\ref{thm:Lstar-tail} (Section~\ref{sec:wll}).
\item[(N4)] \emph{(New.)} Closed-form expressions for $D_C^{\rw}$ in the Gaussian, Poisson, and exponential models (Section~\ref{sec:examples}), and the $M$-ary extension showing that the optimal rate equals the minimum pairwise weighted Chernoff information.
\item[(A1)] \emph{(Adapted definitions.)} The definitions of the weighted Bhattacharyya affinities and the weighted Chernoff information generalise the classical unweighted quantities of \cite{Ch,Hoeffding65} and follow the context-sensitive framework of \cite{KS1,KS2}; their asymptotic and information-geometric consequences developed below are new.
\item[(A2)] \emph{(Geometric context.)} The information-geometric identities of Section~\ref{subsec:bregman} are derived in the spirit of the Chentsov--Amari--Nielsen framework \cite{Chentsov1982,AmariNagaoka,Amari2016,NielsenSPL2013} but are stated and proved for the tilted log-normaliser $\hat F(\theta)=\ln\int\varphi(x)\,e^{\theta^{\rT}t(x)+k(x)}\,\rd\mu(x)$; the unweighted limit $\varphi\equiv 1$ recovers the classical statements of \cite{N,NielsenSPL2013}.
\item[(A3)] \emph{(Standard tool.)} The concentration argument uses the Azuma--Hoeffding/~McDiarmid inequality \cite{A,McD}; the novelty lies in its application to the tilted weighted log-likelihood.
\end{itemize}}

\subsection{Related Work}
\added{The exponential theory of testing errors goes back to Chernoff \cite{Ch} and Hoeffding \cite{Hoeffding65}. The context-sensitive framework and the weighted information quantities used here were developed in \cite{KS1,KS2}.} {The information-geometric viewpoint on Chernoff information originates with Chentsov \cite{Chentsov1982}; the dually flat structure of exponential and mixture families and the associated $\alpha$-divergences are developed in \cite{AmariNagaoka,Amari2016}, and the Chernoff point is characterised as the intersection of an exponential geodesic with the Kullback--Leibler bisector in \cite{NielsenSPL2013}. For $\varphi\equiv 1$, the likelihood-ratio exponential family description is given in \cite{N}; the present paper extends this picture to the tilted integrand $\varphi\,p^\alpha q^{1-\alpha}$. The minimum-pairwise principle for multiple testing is due to \cite{Nielsen2013}. Weighting mechanisms for covariate-dependent relevance have also been studied outside the asymptotic error-exponent framework, e.g.,\ adaptive-kernel conditional-independence testing \cite{RenArtInt2025}.}

\subsection{Structure of the Paper}
\added{Section~\ref{sec:setup} introduces the weighted Bhattacharyya affinities and the weighted Chernoff information. Section~\ref{sec:main-results} proves the main asymptotic result \eqref{eq:intro-main-asymptotic} and develops the exponential-family and information-geometric identities. Subsection~\ref{sec:wll} studies the tilted weighted log-likelihood and derives finite-$n$ concentration bounds. Section~\ref{sec:examples} examines Gaussian, Poisson, and exponential models and includes the $M$-ary extension. Auxiliary computations are collected in the appendices.}

\section{Problem Set-Up and Weighted Divergences}\label{sec:setup}

\subsection{Context-Sensitive Losses and Weighted Total Variation}\label{subsec:setup-loss-tv}

We keep the binary i.i.d.\ model from Section~\ref{sec:intro} and work under
Assumption~\ref{ass:phi-factorised}. 
In particular,
\[
p(x_1^n)=\prod_{i=1}^n p(x_i),\qquad q(x_1^n)=\prod_{i=1}^n q(x_i),
\qquad
\varphi(x_1^n)=\prod_{i=1}^n \varphi(x_i).
\]
Define the $\varphi$-tilted (reweighted) densities
\[
p^*(x):=\frac{\varphi(x)p(x)}{E_\varphi(p)},\qquad
q^*(x):=\frac{\varphi(x)q(x)}{E_\varphi(q)},\qquad
E_\varphi(p):=\int \varphi(x)p(x)\,d\mu(x),
\]
(and similarly for $E_\varphi(q)$). Throughout this section, we assume that $E_{\varphi}(p),E_{\varphi}(q)\in(0,\infty)$.
(Equivalently, $\rho^{\rw}_0(p,q), \rho^{\rw}_1(p,q)\in (0,\infty).$)
Then $p^*,q^*$ are probability densities and, under
$\varphi(x_1^n)=\prod\limits_{i=1}^n\varphi(x_i)$, we have $\varphi(x_1^n)p(x_1^n)=E_\varphi(p)^n 
(p^*)^{\otimes n}(x_1^n)$ (and similarly for $q$).

Under Assumption~\ref{ass:phi-factorised}, we have
\[
\int_{{\cal X}^n}\varphi(x_1^n)\,p(x_1^n)\,d\mu^{\otimes n}=(E_{\varphi}(p))^n,
\qquad
\int_{{\cal X}^n}\varphi(x_1^n)\,q(x_1^n)\,d\mu^{\otimes n}=(E_{\varphi}(q))^n.
\]

Let $\mathcal D$ denote the class of (possibly randomised) decision rules
$D:{\cal X}^n\to[0,1]$, where $D(x_1^n)$ is the probability of deciding in favour of $H_1$
after observing $x_1^n$. (Deterministic rules correspond to $D\in\{0,1\}$.)

For $D\in\mathcal D$, define the context-sensitive type-I and type-II losses by
\begin{align}
\alpha_{\varphi}(D)
&:=\bE_{\bP^{\otimes n}}\!\big[\varphi(X_1^n)\,D(X_1^n)\big]
=\int_{{\cal X}^n}\varphi(x_1^n)\,D(x_1^n)\,p(x_1^n)\,\rd\mu^{\otimes n}(x_1^n),
\label{eq:alpha-varphi-def}\\
\beta_{\varphi}(D)
&:=\bE_{\bQ^{\otimes n}}\!\big[\varphi(X_1^n)\,(1-D(X_1^n))\big]
=\int_{{\cal X}^n}\varphi(x_1^n)\,(1-D(x_1^n))\,q(x_1^n)\,\rd\mu^{\otimes n}(x_1^n),
\label{eq:beta-varphi-def}
\end{align}
and the corresponding total loss
\[
L_n(D):=\alpha_{\varphi}(D)+\beta_{\varphi}(D),
\qquad
L_n^*:=\inf_{D\in\mathcal D} L_n(D).
\]

\begin{Proposition}[Pointwise form of the optimal total loss]\label{prop:pointwise-opt-loss}
For each $n\ge 1$,
\begin{equation}\label{eq:opt-loss-integral}
L_n^*
=\int_{{\cal X}^n}\varphi(x_1^n)\min\{p(x_1^n),q(x_1^n)\}\,\rd\mu^{\otimes n}(x_1^n).
\end{equation}
Moreover, an optimal (deterministic) decision rule is given by the likelihood-ratio test
\[
D_n^*(x_1^n)={\bf 1}\{q(x_1^n)\ge p(x_1^n)\}
\]
(with any measurable tie-breaking on $\{p=q\}$).
\end{Proposition}

\begin{proof}
Fix $x_1^n$. The integrand in $L_n(D)$ equals
\[
\varphi(x_1^n)\big(p(x_1^n)D(x_1^n)+q(x_1^n)(1-D(x_1^n))\big)
=\varphi(x_1^n)\big(q(x_1^n)+D(x_1^n)(p(x_1^n)-q(x_1^n))\big).
\]
Minimising pointwise over $D(x_1^n)\in[0,1]$ yields $D_n^*(x_1^n)=1$ when $p(x_1^n)\le q(x_1^n)$
and $D_n^*(x_1^n)=0$ when $p(x_1^n)>q(x_1^n)$, giving \eqref{eq:opt-loss-integral}.
\end{proof}

We also use the weighted total variation distance
\begin{equation}\label{eq:TVw-def}
{\rm TV}_{\varphi}(\bP^{\otimes n},\bQ^{\otimes n})
:=\frac12\int_{{\cal X}^n}\varphi(x_1^n)\,\big|p(x_1^n)-q(x_1^n)\big|\,\rd\mu^{\otimes n}(x_1^n).
\end{equation}
\begin{remark}

For $\varphi\equiv 1$, this reduces to the usual total variation distance.
If $\varphi$ vanishes on a non-negligible set, ${\rm TV}_{\varphi}$ is, in general,
a pseudo-distance; this is sufficient for our purposes since it characterises the weighted losses.
\end{remark}

Using $\min\{a,b\}=\frac12(a+b-|a-b|)$ in \eqref{eq:opt-loss-integral} and the definition of ${\rm TV}_{\varphi}$ yields
\begin{equation}\label{eq:TVw-loss-identity}
L_n^*=\frac12\Big((E_{\varphi}(p))^n+(E_{\varphi}(q))^n\Big)-{\rm TV}_{\varphi}(\bP^{\otimes n},\bQ^{\otimes n}).
\end{equation}

\subsection{Weighted Affinities and Chernoff Information}\label{subsec:setup-div}

We introduce the weighted Bhattacharyya affinities and the weighted Chernoff information.
Assume that $\rho^{\rw}_\alpha(p,q)\in(0,\infty)$ for all $\alpha\in[0,1]$.

\begin{definition}[Weighted Bhattacharyya coefficient and distance]\label{def:weighted-bhat}
For $\alpha\in[0,1]$ define the weighted $\alpha$-skewed Bhattacharyya affinity coefficient
\begin{equation}\label{eq:rho-alpha-w}
\rho^{\rw}_{\alpha}(p,q):=\int_{\cal X} \varphi(x)\, p(x)^{\alpha} q(x)^{1-\alpha}\,\rd\mu(x),
\end{equation}
and the corresponding weighted Bhattacharyya distance
\begin{equation}\label{eq:Db-alpha-w}
D^{\rw}_{B,\alpha}(p,q):=-\ln \rho^{\rw}_{\alpha}(p,q).
\end{equation}
\end{definition}

\begin{definition}[Weighted Chernoff information]\label{def:weighted-chernoff}
The weighted Chernoff information divergence between $\bP$ and $\bQ$ is
\begin{equation}\label{eq:DCw-def}
D^{\rw}_C({\bP}, {\bQ})=\max_{\alpha\in [0,1]}\left[-\ln \int_{\cal X} \varphi(x)\, p(x)^{\alpha}q(x)^{1-\alpha}\,\rd\mu(x)\right]
=\max_{\alpha\in[0,1]} D^{\rw}_{B,\alpha}(p,q).
\end{equation}
A maximiser $\alpha^*=\alpha^*(p,q)$ in \eqref{eq:DCw-def} is called the optimal Chernoff parameter.
\end{definition}

\begin{remark}\label{rem:chernoff-symmetry}
The weighted Chernoff information is symmetric:
$D_C^{\rw}(\bP,\bQ)=D_C^{\rw}(\bQ,\bP)$, since
$\rho^{\rw}_{\alpha}(p,q)=\rho^{\rw}_{1-\alpha}(q,p)$.
In general, however, $D_C^{\rw}$ does not satisfy the triangle inequality and is therefore a divergence rather than a metric.
\end{remark}
\begin{remark}
\label{rem:single-letter}
Under Assumption~\ref{ass:phi-factorised}, for every $\alpha\in[0,1]$,
\[
\int_{{\cal X}^n}\varphi(x_1^n)\,p(x_1^n)^{\alpha}q(x_1^n)^{1-\alpha}\,\rd\mu^{\otimes n}(x_1^n)
=\bigl(\rho^{\rw}_{\alpha}(p,q)\bigr)^n.
\]
Consequently, the weighted Bhattacharyya distances are additive in $n$ and the corresponding
Chernoff exponent is of single-letter form.
\end{remark}

\section{{Asymptotics and Information-Geometric Identities}}\label{sec:main-results}

{Before stating the main theorem, we separate the two optimisations that appear throughout this section and should not be conflated. The total loss $L_n(D)=\alpha_\varphi(D)+\beta_\varphi(D)$ is a non-negative functional of the decision rule and is \emph{minimised} over $D\in\mathcal D$, giving the optimal loss $L_n^*=\inf_{D\in\mathcal D} L_n(D)$. The map $\alpha\mapsto-\ln\rho^{\rw}_\alpha(p,q)$ is a non-negative concave functional of the skewing parameter and is \emph{maximised} over $\alpha\in[0,1]$, giving the weighted Chernoff information $D_C^{\rw}(\bP,\bQ)=\sup_{\alpha\in[0,1]}[-\ln\rho^{\rw}_\alpha(p,q)]$. Theorem~\ref{thm:optimal-sum-loss} below connects the two via the single-letter asymptotic $L_n^*=\exp\{-nD_C^{\rw}(\bP,\bQ)+o(n)\}$.}

\subsection{Asymptotics of the Optimal Sum of Losses}\label{subsec:main-loss}

Recall from Section~\ref{sec:setup} that
\[
L_n^*:=\inf_{D\in\mathcal D}\bigl[\alpha_\varphi(D)+\beta_\varphi(D)\bigr],
\]
and that Proposition~\ref{prop:pointwise-opt-loss} yields \eqref{eq:opt-loss-integral}.
The next theorem identifies its exact logarithmic asymptotic rate with the weighted Chernoff
information from Definition~\ref{def:weighted-chernoff}.

\begin{Theorem}[Optimal sum of context-sensitive losses]\label{thm:optimal-sum-loss}
Consider the binary hypotheses
$H_0:\ X_1^n\sim \bP^{\otimes n}$ versus $H_1:\ X_1^n\sim \bQ^{\otimes n}$,
under Assumption~\ref{ass:phi-factorised}. Assume also
\begin{equation}\label{finitemean}
\sup\limits_{\alpha\in [0,1]}\int \varphi(x) |\ln\frac{p(x)}{q(x)}|p(x)^{\alpha}q(x)^{1-\alpha}{\rd}\mu(x)<\infty.
\end{equation}
Let $\mathcal D$ be the class of (possibly randomised) decision rules
$D:{\cal X}^n\to[0,1]$, and let $\alpha_\varphi(D)$, $\beta_\varphi(D)$ be defined by
\eqref{eq:alpha-varphi-def}--\eqref{eq:beta-varphi-def}.
Assume that $p,q>0$ $\mu$-a.e.\ and that $\rho^{\rw}_\alpha(p,q)\in(0,\infty)$ for all $\alpha\in[0,1]$.

Then, as $n\to\infty$,
\begin{equation}\label{eq:optimal-sum-loss}
L_n^*=\exp\{-nD_C^{\rw}(\bP,\bQ)+o(n)\}.
\end{equation}
Equivalently,
\[
\lim_{n\to\infty}-\frac{1}{n}\ln L_n^* = D_C^{\rw}(\bP,\bQ).
\]
\end{Theorem}

\begin{proof}
\added{By Proposition~\ref{prop:pointwise-opt-loss},
\[
L_n^*=\int_{{\cal X}^n}\varphi(x_1^n)\min\{p(x_1^n),q(x_1^n)\}\,\rd\mu^{\otimes n}(x_1^n).
\]
For any $\alpha\in[0,1]$, $\min\{a,b\}\le a^\alpha b^{1-\alpha}$, hence, by factorisation,
\[
L_n^*\le \bigl(\rho^{\rw}_\alpha(p,q)\bigr)^n.
\]
Taking the infimum over $\alpha\in[0,1]$ gives
\[
\liminf_{n\to\infty}-\frac1n\ln L_n^*\ge D_C^{\rw}(\bP,\bQ).
\]

Now fix $\alpha^*\in\arg\min_{\alpha\in[0,1]}\rho^{\rw}_\alpha(p,q)$ and define
\[
r_{\alpha^*}(x):=\frac{\varphi(x)p(x)^{\alpha^*}q(x)^{1-\alpha^*}}{\rho^{\rw}_{\alpha^*}(p,q)},
\qquad
S_n:=\sum_{i=1}^n\ln\frac{p(X_i)}{q(X_i)}.
\]
A direct change of measure yields
\[
L_n^*
=\bigl(\rho^{\rw}_{\alpha^*}(p,q)\bigr)^n
\bE_{r_{\alpha^*}^{\otimes n}}\!\Big[
e^{-\alpha^*S_n}\mathbf{1}\{S_n>0\}
+e^{(1-\alpha^*)S_n}\mathbf{1}\{S_n\le 0\}
\Big].
\]
The bracket is bounded above by $1$.

Let $F(\alpha):=\ln\rho^{\rw}_\alpha(p,q)$, it is easy to check that
$F(\alpha)$ is a convex function.
Under the regularity assumption of (\ref{finitemean}),
\[
F'(\alpha)=\bE_{r_\alpha}\!\left[\ln\frac{p(X)}{q(X)}\right].
\]
If $\alpha^*\in(0,1)$, then $F'(\alpha^*)=0$. Hence, in view of LLN, $S_n/n\to 0$ in $r_{\alpha^*}$-probability. Therefore, for every $\varepsilon>0$,
\[
\bE_{r_{\alpha^*}^{\otimes n}}\!\Big[
e^{(1-\alpha^*)S_n}\mathbf{1}\{S_n \leq 0\}
+e^{-\alpha^*S_n}\mathbf{1}\{S_n > 0\}
\Big]
\ge e^{-\varepsilon n}\,
r_{\alpha^*}^{\otimes n}\big(|S_n|\le \varepsilon n\big)
=\exp\{o(n)\}.
\]
Combined with the upper bound by $1$, this implies
\[
L_n^*=\exp\{-nD_C^{\rw}(\bP,\bQ)+o(n)\}.
\]

In the boundary case $\alpha^*=0$, we have the mean value $m>0$, and
$\mathbf{1}(S_n>0)e^{-\alpha^* S_n}\vert_{\alpha^*=0}\to 1$ a.s. as $n\to\infty$ by the strong
LLN. Similarly, in the case $\alpha^*=1$, we have the mean value $m\leq 0$, and
$\mathbf{1}(S_n\leq 0)e^{(1-\alpha^*)S_n}\vert_{\alpha^*=1}\to 1$ a.s.
as $n\to\infty$. This completes the proof.}
\end{proof}

\begin{Corollary}[Asymptotics of the weighted total variation]\label{cor:TVw-asympt}
Under the assumptions of Theorem~\ref{thm:optimal-sum-loss}, the weighted total variation satisfies
\begin{equation}\label{eq:TVw-asympt}
{\rm TV}_{\varphi}(\bP^{\otimes n},\bQ^{\otimes n})
=\frac{1}{2}\Big((E_{\varphi}(p))^n+(E_{\varphi}(q))^n\Big)
-\exp\{-nD^{\rw}_{C}(\bP,\bQ)+o(n)\},~~ n\to\infty,
\end{equation}
where $E_{\varphi}(r)=\int_{\cal X}\varphi(x)r(x)\,\rd\mu(x)$.
\end{Corollary}

\begin{proof}
Combine the identity \eqref{eq:TVw-loss-identity} with \eqref{eq:optimal-sum-loss}.
\end{proof}

\begin{remark}\label{rem:bhat-chernoff}
The weighted $\alpha$-skewed Bhattacharyya distance \eqref{eq:Db-alpha-w} appears in many papers, see e.g.,\ \cite{N}.
Definition~\ref{def:weighted-chernoff} shows that the weighted Chernoff information divergence
is the maximally skewed weighted Bhattacharyya distance.
\end{remark}
\subsection{Exponential-Family Representation and Uniqueness of $\alpha^*$}\label{subsec:exp-family}

In order to develop an effective computational procedure and to connect the weighted Chernoff information to information geometry,
we embed the weighted geometric mixtures of $p$ and $q$ into a one-parameter likelihood-ratio exponential family.
For $\alpha\in[0,1]$ define
\begin{equation}\label{eq:Zpq-def}
Z_{pq}(\alpha)
:=\int_{\cal X}\varphi(x)\,p(x)^{\alpha}q(x)^{1-\alpha}\,\rd\mu(x)
=\rho^{\rw}_{\alpha}(p,q),
\end{equation}
and the corresponding normalised density
\begin{equation}\label{eq:Epq-def}
{\cal E}_{pq}
=\left\{(pq)_{\alpha}(x):=\frac{\varphi(x)\, p(x)^{\alpha}q(x)^{1-\alpha}}{Z_{pq}(\alpha)}\;:\;\alpha\in [0,1]\right\}.
\end{equation}
By assumption $Z_{pq}(\alpha)\in(0,\infty)$, so $(pq)_\alpha$ is well-defined as a probability density w.r.t.\ $\mu$.

Set $t(x):=\ln\frac{p(x)}{q(x)}$ and $k_{pq}(x):=\ln\varphi(x)+\ln q(x)$.
Then, $(pq)_\alpha$ admits the exponential-family form
\begin{align}\label{eq:Epq-expform}
(pq)_{\alpha}(x)
&=\exp\!\left\{\alpha\,t(x)-F_{pq}(\alpha)+k_{pq}(x)\right\},
\\
F_{pq}(\alpha):&=\ln Z_{pq}(\alpha)=-D^{\rw}_{B,\alpha}(p,q).
\end{align}
In particular, $t(X)$ is a sufficient statistic for the family ${\cal E}_{pq}$.

The log-normaliser $F_{pq}$ is convex on $[0,1]$; if $\ln\frac{p}{q}$ is not $\mu$-a.e.\ constant on $\{\varphi>0\}$,
then $F_{pq}$ is strictly convex and the maximiser $\alpha^*$ in \eqref{eq:DCw-def} is unique.
By H\"older's inequality,
\[
Z_{pq}(\alpha)=\int_{\cal X} (\varphi p)^{\alpha}(\varphi q)^{1-\alpha}\,\rd\mu
\le (E_{\varphi}(p))^{\alpha}(E_{\varphi}(q))^{1-\alpha}.
\]
Finally, note that $Z_{pq}(1)=E_{\varphi}(p)$ and $Z_{pq}(0)=E_{\varphi}(q)$; hence,
\[
(pq)_1(x)=\frac{\varphi(x)p(x)}{E_{\varphi}(p)},
\qquad
(pq)_0(x)=\frac{\varphi(x)q(x)}{E_{\varphi}(q)},
\]
so ${\cal E}_{pq}$ is an exponential arc between the tilted versions of $\bP$ and $\bQ$.
By this definition, the following identities hold:
\begin{equation}
	\begin{array}{l}
		D^{\rw}_C(p,q)=D^{\rw}_{B,\alpha^*(p,q)}(p,q)=D^{\rw}_{B,\alpha^*(q,p)}(q,p)=D^{\rw}_C(q,p).
	\end{array}	
\end{equation}

\subsection{Weighted Bregman Divergence and Information-Geometric Identities}\label{subsec:bregman}

\subsubsection{\highlighting{Weighted}  KL Divergence and Weighted Bregman Divergence} This subsection collects information-geometric identities useful for analysing
$\rho^{\rw}_\alpha$ and for computing the optimal Chernoff parameter $\alpha^*$.
We follow \cite{KS1} for weighted Bregman divergences.

Let ${\cal E}=\{p_\theta:\theta\in\Theta\subset\R^d\}$ be a regular exponential family of densities
(with respect to $\mu$),
\begin{equation}\label{eq:expfam-def}
p_\theta(x)=\exp\{\theta^{\rT}t(x)-F(\theta)+k(x)\},\qquad x\in{\cal X}.
\end{equation}
For a density $r$, set $E_\varphi(r):=\int_{\cal X}\varphi(x)\,r(x)\,\rd\mu(x)$ and write
$E_\varphi(\theta):=E_\varphi(p_\theta)$.

Assume $E_\varphi(\theta)\in(0,\infty)$ for $\theta\in\Theta$ and define the tilted log-normaliser
\begin{equation}\label{eq:Fhat-def}
\hat F(\theta)
:=\ln\int_{\cal X}\varphi(x)\,e^{\theta^{\rT}t(x)+k(x)}\,\rd\mu(x)
=F(\theta)+\ln E_\varphi(\theta).
\end{equation}
Equivalently, the tilted density is
\begin{equation}\label{eq:adjoint-density}
p_\theta^*(x)=\frac{\varphi(x)\,p_\theta(x)}{E_\varphi(\theta)}.
\end{equation}

\begin{definition}[Weighted Kullback--Leibler divergence]\label{def:wkl}
For densities $p,q$ on ${\cal X}$ define
\begin{equation}\label{eq:KLw-def}
D^{\rw}_{\rm KL}(p\|q)
:=\int_{\cal X}\varphi(x)\,p(x)\,\ln\frac{p(x)}{q(x)}\,\rd\mu(x),
\end{equation}
whenever the integral is well defined in $(-\infty,\infty]$.
\end{definition}

\begin{definition}[Weighted Bregman divergence]\label{def:wbreg}
The weighted Bregman divergence associated with $(F,\hat F)$ is
\begin{align}
B^{\rw}_{\varphi,F}(\theta_1,\theta_2)
:&=e^{\hat F(\theta_2)-F(\theta_2)}
\Big[F(\theta_1)-F(\theta_2)-(\theta_1-\theta_2)^{\rT}\nabla \hat F(\theta_2)\Big]
\label{eq:weighted-bregman}
\\
&=E_\varphi(\theta_2)\Big[F(\theta_1)-F(\theta_2)-(\theta_1-\theta_2)^{\rT}\nabla \hat F(\theta_2)\Big]. \nonumber
\end{align}
\end{definition}


\begin{Proposition}[Weighted KL as weighted Bregman divergence]\label{prop:kl-bregman}
For a regular exponential family ${\cal E}=\{\bP_{\theta}, \theta\in \Theta\}$,
assume that the integral in \eqref{eq:KLw-def} is well-defined in $(-\infty,\infty]$.
Then for any $\theta_1,\theta_2\in\Theta$,
\begin{equation}\label{eq:kl-bregman}
D^{\rw}_{\rm KL}(p_{\theta_1}\|p_{\theta_2})
=B^{\rw}_{\varphi,F}(\theta_2,\theta_1).
\end{equation}
\end{Proposition}

\begin{proof}
This identity is stated in \cite[Proposition~4.1]{KS1}; we give a short derivation for completeness.
By \eqref{eq:expfam-def},
\[
\ln\frac{p_{\theta_1}(x)}{p_{\theta_2}(x)}
=(\theta_1-\theta_2)^{\rT}t(x)-(F(\theta_1)-F(\theta_2)).
\]
Substituting this into \eqref{eq:KLw-def} yields
\[
D^{\rw}_{\rm KL}(p_{\theta_1}\|p_{\theta_2})
=(\theta_1-\theta_2)^{\rT}\!\int_{\cal X}\varphi(x)t(x)p_{\theta_1}(x)\,\rd\mu(x)
-(F(\theta_1)-F(\theta_2))\,E_\varphi(\theta_1).
\]
Using \eqref{eq:Fhat-def} and differentiation under the integral sign (regularity of ${\cal E}$),
\[
\nabla \hat F(\theta_1)
=\frac{\int_{\cal X}\varphi(x)t(x)e^{\theta_1^{\rT}t(x)+k(x)}\,\rd\mu(x)}
{\int_{\cal X}\varphi(x)e^{\theta_1^{\rT}t(x)+k(x)}\,\rd\mu(x)}
=\frac{\int_{\cal X}\varphi(x)t(x)p_{\theta_1}(x)\,\rd\mu(x)}{E_\varphi(\theta_1)}.
\]
Hence, $\int \varphi t p_{\theta_1}\,\rd\mu = E_\varphi(\theta_1)\nabla\hat F(\theta_1)$, and therefore
\[
D^{\rw}_{\rm KL}(p_{\theta_1}\|p_{\theta_2})
=E_\varphi(\theta_1)\Big[F(\theta_2)-F(\theta_1)-(\theta_2-\theta_1)^{\rT}\nabla \hat F(\theta_1)\Big]
=B^{\rw}_{\varphi,F}(\theta_2,\theta_1),
\]
which proves \eqref{eq:kl-bregman}.
\end{proof}

\begin{Proposition}[Primal--dual identities for (weighted) Bregman divergences]\label{prop:weighted-bregman-duality}
Let $F$ be a log-normaliser of a regular exponential family and let $F^*$ denote its Legendre transform.
Write $\theta^*=\nabla F(\theta)$ and $\theta=\nabla F^*(\theta^*)$.

\smallskip\noindent
(a) \emph{Weighted one-parameter identity.}
Assume $d=1$ (one-parameter case) and let $\theta_i^*:=F'(\theta_i)$. Then, the following weighted analogue
of the classical Bregman duality holds:
\begin{equation}\label{eq:weighted-bregman-duality}
B^{\rw}_{\varphi,F}(\theta_1,\theta_2)
=
B^{\rw}_{\varphi,F^*}(\theta_2^*,\theta_1^*)
-(\theta_1-\theta_2)\,\ln E_\varphi(\theta_2)
+(\theta_2^*-\theta_1^*)\,\ln E_\varphi(\theta_1),
\end{equation}
where $B^{\rw}_{\varphi,F}$ is as in Definition~\ref{def:wbreg} and $B^{\rw}_{\varphi,F^*}$ is defined analogously
(with $F$ replaced by $F^*$ and with the convention $E_\varphi(\theta^*):=E_\varphi(\theta)$ under $\theta=\nabla F^*(\theta^*)$).

\smallskip\noindent
(b) \emph{Classical identity.}
For any $d\ge 1$, the (unweighted) Bregman divergence admits the standard Legendre representation
\begin{equation}\label{eq:bregman-legendre}
B_F(\theta_0,\theta_1)=F(\theta_0)+F^*(\theta_1^*)-\theta_0^{\rT}\theta_1^*,
\qquad \theta_1^*=\nabla F(\theta_1),
\end{equation}
where $B_F(\theta_0,\theta_1)=F(\theta_0)-F(\theta_1)-(\theta_0-\theta_1)^{\rT}\nabla F(\theta_1)$
is the usual (unweighted) Bregman divergence.
\end{Proposition}

\begin{proof}
Part (a) is a weighted extension of the classical duality $B_F(\theta_1,\theta_2)=B_{F^*}(\theta_2^*,\theta_1^*)$
and follows by combining the weighted representation \eqref{eq:weighted-bregman} with Legendre relations;
see also \cite{KS1}. Part (b) is standard.
\end{proof}

\subsubsection{Weighted Chernoff/Bhattacharyya Quantities Inside an Exponential Family}
Let $p_{\theta_1},p_{\theta_2}\in{\cal E}$ and $\theta_\alpha:=\alpha\theta_1+(1-\alpha)\theta_2$.
A direct calculation yields
\begin{align}\label{eq:rho-expfam}
\rho^{\rw}_{\alpha}(p_{\theta_1},p_{\theta_2})
& =\ln\int_{\cal X}\varphi(x)\,p_{\theta_1}(x)^\alpha p_{\theta_2}(x)^{1-\alpha}\,\rd\mu(x)
\\
& =\hat F(\theta_\alpha)-\alpha F(\theta_1)-(1-\alpha)F(\theta_2). 
\nonumber
\end{align}
Consequently,
\begin{align}\label{eq:Db-expfam}
D^{\rw}_{B,\alpha}(p_{\theta_1},p_{\theta_2})
& =\alpha F(\theta_1)+(1-\alpha)F(\theta_2)-\hat F(\theta_\alpha) 
\\
& =U_{F,\alpha}(\theta_1,\theta_2)-\ln E_\varphi(\theta_\alpha),
\nonumber
\end{align}
where $U_{F,\alpha}(\theta_1,\theta_2)
:=\alpha F(\theta_1)+(1-\alpha)F(\theta_2)-F(\theta_\alpha)$
is the (unweighted) Jensen/Burbea--Rao divergence induced by $F$.
In particular, when $\varphi\equiv 1$ we have $\hat F\equiv F$ and
$D^{\rw}_{B,\alpha}(p_{\theta_1},p_{\theta_2})=U_{F,\alpha}(\theta_1,\theta_2)$.

\begin{remark}[Geometric mixtures and tilting by $\varphi$]\label{rem:geom-mixture}
In particular, when $\varphi\equiv 1$, we have $\hat F\equiv F$ and the normalised geometric mixture
$p_{\theta_1}^\alpha p_{\theta_2}^{1-\alpha}$ belongs to the same exponential family, namely,
$\propto p_{\theta_\alpha}$ with $\theta_\alpha=\alpha\theta_1+(1-\alpha)\theta_2$.
\end{remark}

\begin{Proposition}[Optimal Chernoff parameter in an exponential family]\label{prop:alpha-star-expfam}
Assume that $\hat F$ is strictly convex on the segment $[\theta_1,\theta_2]$ and that the maximiser
$\alpha^*\in(0,1)$ exists. Then, $\alpha^*$ is unique and satisfies
\begin{equation}\label{eq:alpha-star-opt}
(\theta_1-\theta_2)^{\rT}\nabla\hat F(\theta_{\alpha^*})=F(\theta_1)-F(\theta_2),
\qquad
\theta_{\alpha^*}=\alpha^*\theta_1+(1-\alpha^*)\theta_2,
\end{equation}
with $\nabla\hat F(\theta)=\bE_{p_\theta^*}[t(X)]$.
\end{Proposition}

\begin{proof}
Differentiate \eqref{eq:Db-expfam} with respect to $\alpha$ and use
$\frac{\rd}{\rd\alpha}\theta_\alpha=\theta_1-\theta_2$.
Strict convexity of $\hat F$ on $[\theta_1,\theta_2]$ implies strict concavity of
$\alpha\mapsto D^{\rw}_{B,\alpha}(p_{\theta_1},p_{\theta_2})$, hence uniqueness.
\end{proof}

\begin{Proposition}[Chernoff information as a Jensen-type divergence and a Bregman bisector]\label{prop:chernoff-jensen-bisector}
Let $p_{\theta_1},p_{\theta_2}\in{\cal E}$ and assume that the maximiser $\alpha^*\in(0,1)$ in
Definition~\ref{def:weighted-chernoff} exists and is unique. Set $\theta_\alpha=\alpha\theta_1+(1-\alpha)\theta_2$. Then
\begin{equation}\label{eq:chernoff-jensen-value}
D_C^{\rw}(p_{\theta_1},p_{\theta_2})
=D^{\rw}_{B,\alpha^*}(p_{\theta_1},p_{\theta_2})
=\alpha^*F(\theta_1)+(1-\alpha^*)F(\theta_2)-\hat F(\theta_{\alpha^*}).
\end{equation}
Moreover, $\theta_{\alpha^*}$ is characterised by the weighted Bregman bisector condition
\begin{equation}\label{eq:chernoff-bisector}
B^{\rw}_{\varphi,F}(\theta_1,\theta_{\alpha^*})=B^{\rw}_{\varphi,F}(\theta_2,\theta_{\alpha^*}),
\end{equation}
and the common value recovers the Chernoff information as
\begin{align}\label{eq:chernoff-bregman-point}
D_C^{\rw}(p_{\theta_1},p_{\theta_2})
&=\frac{1}{E_\varphi(\theta_{\alpha^*})}B^{\rw}_{\varphi,F}(\theta_1,\theta_{\alpha^*})-\ln E_\varphi(\theta_{\alpha^*})
\\
&=\frac{1}{E_\varphi(\theta_{\alpha^*})}B^{\rw}_{\varphi,F}(\theta_2,\theta_{\alpha^*})-\ln E_\varphi(\theta_{\alpha^*}).
\end{align}
In the special case $\varphi\equiv 1$, we have $\hat F\equiv F$ and $E_\varphi(\theta)\equiv 1$, so that
\eqref{eq:chernoff-jensen-value} reduces to the classical Jensen divergence induced by $F$ and
\eqref{eq:chernoff-bregman-point} becomes $D_C(p_{\theta_1},p_{\theta_2})=B_F(\theta_1,\theta_{\alpha^*})=B_F(\theta_2,\theta_{\alpha^*})$.
\end{Proposition}

\begin{proof}
Equation \eqref{eq:chernoff-jensen-value} is \eqref{eq:Db-expfam} at $\alpha=\alpha^*$.
For \eqref{eq:chernoff-bisector}, expand
$B^{\rw}_{\varphi,F}(\theta_i,\theta_{\alpha^*})=E_\varphi(\theta_{\alpha^*})\big[F(\theta_i)-F(\theta_{\alpha^*})-(\theta_i-\theta_{\alpha^*})^{\rT}\nabla\hat F(\theta_{\alpha^*})\big]$
and use \eqref{eq:alpha-star-opt} to see that the difference vanishes.
Finally, substituting 
$\theta_1-\theta_{\alpha^*}=(1-\alpha^*)(\theta_1-\theta_2)$
into
$B^{\rw}_{\varphi,F}(\theta_1,\theta_{\alpha^*})/E_\varphi(\theta_{\alpha^*})$ and using \eqref{eq:alpha-star-opt} gives
\[
\frac{1}{E_\varphi(\theta_{\alpha^*})}B^{\rw}_{\varphi,F}(\theta_1,\theta_{\alpha^*})
=\alpha^*F(\theta_1)+(1-\alpha^*)F(\theta_2)-{F}(\theta_{\alpha^*}).
\]
Since $\hat F(\theta_{\alpha^*})=F(\theta_{\alpha^*})+\ln E_\varphi(\theta_{\alpha^*})$, this yields \eqref{eq:chernoff-bregman-point}.
\end{proof}

\subsubsection{Derivative and Weighted KL}
Recall $F_{pq}(\alpha)=\ln Z_{pq}(\alpha)=\ln\rho^{\rw}_\alpha(p,q)$.
In view of (\ref{finitemean} ), the  differentiation under the integral sign is justified, and we have
\begin{equation}\label{eq:Fpq-derivative}
F_{pq}'(\alpha)
=\bE_{(pq)_\alpha}\!\left[\ln\frac{p(X)}{q(X)}\right],
\end{equation}
where $(pq)_\alpha$ is the Chernoff-tilted density from \eqref{eq:Epq-def}.
In particular,
\begin{equation}\label{eq:KL-derivative}
F_{pq}'(1)
=\frac{1}{E_\varphi(p)}\int_{\cal X}\varphi(x)\,p(x)\,\ln\frac{p(x)}{q(x)}\,\rd\mu(x)
=\frac{1}{E_\varphi(p)}D^{\rw}_{\rm KL}(p\|q).
\end{equation}
Analogously, $F'_{pq}(0)= -\frac{1}{E_\varphi(q)}D^{\rw}_{\rm KL}(q\|p)$.


\subsubsection{Chernoff--KL.}
\begin{Lemma}\label{lem:chernoff-kl}
Let $\alpha^*$ be a maximiser in Definition~\ref{def:weighted-chernoff} and assume that
$\alpha^*\in(0,1)$ (so that $F_{pq}'(\alpha^*)=0$ below). Set
\(
F_{pq}(\alpha):=\ln\rho^{\rw}_\alpha(p,q), r_\alpha:=(pq)_\alpha.
\)
Then
\begin{equation}\label{eq:chernoff-kl}
D^{\rw}_C(p,q)
= D_{\rm KL}\!\big(r_{\alpha^*}\,\|\,r_1\big)-\ln E_{\varphi}(p)
= D_{\rm KL}\!\big(r_{\alpha^*}\,\|\,r_0\big)-\ln E_{\varphi}(q),
\end{equation}
where $r_1(x)=\varphi(x)p(x)/E_\varphi(p)$ and $r_0(x)=\varphi(x)q(x)/E_\varphi(q)$.
\end{Lemma}
Here, $D_{\rm KL}$ denotes the standard (unweighted) Kullback–Leibler divergence.
\begin{proof}
A direct computation yields, for $\alpha\in[0,1]$,
\begin{align*}
D_{\rm KL}(r_\alpha\|r_1)&=-(1-\alpha)F'_{pq}(\alpha)-F_{pq}(\alpha)+F_{pq}(1),
\\
D_{\rm KL}(r_\alpha\|r_0)&=\alpha F'_{pq}(\alpha)-F_{pq}(\alpha)+F_{pq}(0).
\end{align*}
At $\alpha=\alpha^*$, we have $F'_{pq}(\alpha^*)=0$. Since $F_{pq}(1)=\ln E_\varphi(p)$ and
$F_{pq}(0)=\ln E_\varphi(q)$, the claim follows from $D_C^{\rw}(p,q)=-F_{pq}(\alpha^*)$.
\end{proof}

\begin{Corollary}[Chernoff information as a Bregman divergence on the Chernoff arc]\label{cor:chernoff-bregman-arc}
Let $F_{pq}(\alpha):=\ln\rho^{\rw}_{\alpha}(p,q)$ and define the one-dimensional Bregman divergence
\[
B_{F_{pq}}(a,b):=F_{pq}(a)-F_{pq}(b)-(a-b)F'_{pq}(b).
\]
Assume that the maximiser $\alpha^*\in(0,1)$ in Definition~\ref{def:weighted-chernoff} is interior,
so that $F'_{pq}(\alpha^*)=0$. Then
\begin{equation}\label{eq:chernoff-bregman-arc}
D_C^{\rw}(p,q)
=B_{F_{pq}}(1,\alpha^*)-\ln E_\varphi(p)
=B_{F_{pq}}(0,\alpha^*)-\ln E_\varphi(q).
\end{equation}
Equivalently,
\begin{align*}
B_{F_{pq}}(1,\alpha^*)&=D_{\rm KL}(r_{\alpha^*}\|r_1)=D_C^{\rw}(p,q)+\ln E_\varphi(p),
\\
B_{F_{pq}}(0,\alpha^*)&=D_{\rm KL}(r_{\alpha^*}\|r_0)=D_C^{\rw}(p,q)+\ln E_\varphi(q),
\end{align*}
with $r_\alpha=(pq)_\alpha$, $r_1=\varphi p/E_\varphi(p)$ and $r_0=\varphi q/E_\varphi(q)$.
\end{Corollary}


\begin{proof}
In the Chernoff exponential family $\{r_\alpha\}$ with log-normalizer $F_{pq}$, the KL--Bregman identity gives
$D_{\rm KL}(r_{\alpha^*}\|r_1)=B_{F_{pq}}(1,\alpha^*)$ and $D_{\rm KL}(r_{\alpha^*}\|r_0)=B_{F_{pq}}(0,\alpha^*)$.
Since $F'_{pq}(\alpha^*)=0$ and $F_{pq}(1)=\ln E_\varphi(p)$, $F_{pq}(0)=\ln E_\varphi(q)$, we obtain
\[
B_{F_{pq}}(1,\alpha^*)=-F_{pq}(\alpha^*)+F_{pq}(1)=D_C^{\rw}(p,q)+\ln E_\varphi(p),
\]
and similarly for $0$. Substituting this into Lemma~\ref{lem:chernoff-kl} yields \eqref{eq:chernoff-bregman-arc}.
\end{proof}

\begin{remark}[One-parameter case]\label{rem:one-parameter-alpha}
Assume $d=1$, $\theta_1\neq\theta_2$, and that the maximiser $\alpha^*\in(0,1)$ in
Definition~\ref{def:weighted-chernoff} is interior.
Assume moreover that $\hat F'$ is strictly increasing on $\Theta$ and set ${\hat G}=(\hat F')^{-1}$.
Then \eqref{eq:alpha-star-opt} yields
\begin{equation}\label{eq:alpha-star-oneparam}
\alpha^*
=\frac{1}{\theta_1-\theta_2}\left({\hat G}\left(\frac{F(\theta_1)-F(\theta_2)}{\theta_1-\theta_2}\right)-\theta_2\right).
\end{equation}
When $\varphi\equiv 1$, we have $\hat F\equiv F$ and \eqref{eq:alpha-star-oneparam}
reduces to the classical formula.
\end{remark}

To illustrate the general identities above, we provide explicit expressions for
$D^{\rw}_{B,\alpha}$ and $D^{\rw}_C$ in several parametric settings (Gaussian, Poisson, and exponential);
see Section~\ref{sec:examples}.

{\begin{remark}
\label{rem:what-identities-yield}
The Bregman representation \eqref{eq:chernoff-bregman-arc} identifies $D_C^{\rw}(\bP,\bQ)$ with a Bregman divergence on the tilted log-normaliser $\hat F$; geometrically, the optimiser $\alpha^\ast$ marks the intersection of the exponential geodesic of the tilted family $\{\varphi\,p^\alpha q^{1-\alpha}\}_{\alpha\in[0,1]}$ with the weighted Kullback--Leibler bisector, generalising the unweighted characterisation of \cite{NielsenSPL2013,N}. In the exponential-family setting, the computation of $D_C^{\rw}$ therefore reduces to $\hat F$ and its gradient: once $\hat F$ is available in closed form, $\alpha^\ast$ is determined by \eqref{eq:alpha-star-oneparam} in the one-parameter case and by a monotone equation in the natural-parameter space in general, without evaluating $\rho_\alpha^{\rw}$ for each $\alpha$. The examples of Section~\ref{sec:examples} exemplify this reduction.
\end{remark}}

\subsection{Tilted Weighted Likelihood and Concentration Bounds}\label{sec:wll}


\added{Although the optimal rule in the context-sensitive problem is still the usual likelihood-ratio test $q/p$ (cf.\ Section~\ref{sec:main-results}), it is convenient to work with the \emph{tilted} ratio $q^*/p^*$. The factor $\varphi$ cancels pointwise in $q^*/p^*$ and enters only through the normalisation constants $E_\varphi(p)$ and $E_\varphi(q)$. We record two consequences: a large-deviation representation for $L^*/n$ via the cumulant generating functions $\psi_{\bP},\psi_{\bQ}$ and their Legendre transforms and a finite-$n$ concentration bound based on a martingale argument.}


For the tilted distributions, the log-likelihood takes the form
\begin{equation}\label{eq:Lstar-def}
\begin{array}{l}
L^*({X}_1^n)=L^*(X_1,\ldots ,X_n)=\sum\limits_{i=1}^n\ln\frac{q^*(X_i)}{p^*(X_i)}\\
\hspace{1.35cm}=\sum\limits_{i=1}^n\ln\frac{q(X_i)}{p(X_i)}-n\ln E_{\varphi}(q)+n\ln E_{\varphi}(p).
\end{array}
\end{equation}
Here $E_{\varphi}(p)=\int \varphi(x)\,p(x)\,\rd\mu(x)$.

In particular, since
\[
\ln\frac{q^*(x)}{p^*(x)}=\ln\frac{q(x)}{p(x)}+\ln E_\varphi(p)-\ln E_\varphi(q),
\]
we may equivalently rewrite likelihood-ratio threshold rules in terms of $L^*$.
For example,
\[
\sum_{i=1}^n\ln\frac{q(X_i)}{p(X_i)}\ge 0
\quad\Longleftrightarrow\quad
L^*(X_1^n)\ge n\big(\ln E_\varphi(p)-\ln E_\varphi(q)\big).
\]
Thus, $L^*$ is the usual log-likelihood ratio, shifted by a constant determined by the context weight $\varphi$.

The log of the moment generating function and its Legendre transform take the form

\begin{align}\label{eq:mgf-legendre}
\psi_{\bP}(\alpha)&=\ln\bE_{\bP}\left[e^{\alpha\ln\frac{q^*(X)}{p^*(X)}}\right]
\nonumber
\\
&=\ln \int_{\cal X} q(x)^{\alpha}p(x)^{1-\alpha}\,\rd\mu(x)-\alpha \ln E_{\varphi}(q)+\alpha \ln E_{\varphi}(p),\\
I_{\bP}(r)&=\sup\limits_{\alpha}\left[\alpha r-\psi_{\bP}(\alpha)\right], \nonumber
\end{align}
where $\alpha$ ranges over the set $\{\alpha\in\mathbb R:\psi_{\bP}(\alpha)<\infty\}$ (and similarly for $\psi_{\bQ}$).

Similarly,
\begin{align}\label{eq:mgf-Q}
\psi_{\bQ}(\alpha):&=\ln\bE_{\bQ}\left[e^{\alpha\ln\frac{q^*(X)}{p^*(X)}}\right]
=\ln\bE_{\bP}\left[e^{\ln\frac{q(X)}{p(X)}+\alpha\ln\frac{q^*(X)}{p^*(X)}}\right]\\
\hspace{1.0cm}&=\ln \int_{\cal X} q(x)^{\alpha+1}p(x)^{-\alpha}\,\rd\mu(x)-\alpha \ln E_{\varphi}(q)+\alpha \ln E_{\varphi}(p).
\nonumber
\end{align}
This implies the relation of Legendre transforms
\begin{equation}\label{eq:legendre-relation}
I_{\bQ}(r)=I_{\bP}(r)-r+\ln E_{\varphi}(p)-\ln E_{\varphi}(q).
\end{equation}
In particular, $I_{\bP}(0)$ may be treated as a natural weighted version of the Chernoff divergence between $q$ and $p$:
\begin{align}\label{eq:IP0}
I_{\bP}(0)&=\sup\limits_{\alpha}\left[-\ln \int_{\cal X} q(x)^{\alpha}p(x)^{1-\alpha}\,\rd\mu(x)+\alpha \ln E_{\varphi}(q)-\alpha \ln E_{\varphi}(p)\right]
\nonumber
\\
&=:{\hat D}^{\rm w}_C(q,p).
\end{align}

\smallskip
\noindent\textit{Interpretation.}
The value $I_{\bP}(0)$ is the Chernoff--Cram\'er exponent controlling the tail event $\{L^*/n\ge 0\}$ under $\bP$,
i.e.,\ (under standard regularity assumptions), $\bP(L^*(X_1^n)\ge 0)\asymp e^{-n I_{\bP}(0)}$.
In the unweighted case $\varphi\equiv 1$, we have $E_\varphi(p)=E_\varphi(q)=1$ and $I_{\bP}(0)$ reduces to the classical
Chernoff information between $p$ and $q$.
We also stress that ${\hat D}_C^{\rm w}(q,p)=I_{\bP}(0)$ is a \emph{tilted-likelihood} exponent and 
is distinct from
 the
weighted Chernoff information $D_C^{\rw}(\bP,\bQ)$ from Definition~\ref{def:weighted-chernoff}, which governs the optimal
sum of context-sensitive losses.

\subsubsection*{Non-Asymptotic Concentration Via a Doob Martingale}
The rate functions $I_{\bP}$ and $I_{\bQ}$ capture the exponential scale of deviations of $L^*/n$ as $n\to\infty$.
To obtain explicit \emph{finite-$n$} bounds, we now apply a standard martingale method to $L^*$ under $\bQ$.

Consider the filtration ${\cal F}_k=\sigma(X_1,\ldots,X_k)$ and define the random variables $\{U_k, k=0,\ldots ,n\}$ by
\begin{align}\label{eq:Uk-def}
U_k & =\bE_{\bQ}\left[L^*(X_1,\ldots ,X_n)\vert {\cal F}_k\right]   \nonumber
\\
&=
\sum\limits_{j=1}^k\ln\frac{q(X_j)}{p(X_j)}+(n-k)D_{\rm KL}(\bQ\|\bP)
-n\ln E_{\varphi}(q)+n\ln E_{\varphi}(p),
\end{align}
where $D_{\rm KL}(\bQ\|\bP)=\int q(x)\ln\frac{q(x)}{p(x)}\,\rd\mu(x)$ stands for the (unweighted) Kullback--Leibler divergence of $\bQ$ and $\bP$. Then
\begin{equation}\label{eq:Uk-increment}
U_k-U_{k-1}=\ln\frac{q(X_k)}{p(X_k)}-D_{\rm KL}(\bQ\|\bP),\qquad k=1,\ldots,n.
\end{equation}
Observe that $\{U_k,{\cal F}_k\}$ is a martingale w.r.t.\ $\bQ$.

\added{Assume now that $d<\infty$ and $\sigma^2<\infty$, where}
$d=\sup\limits_{x\in {\cal X}}\left| \ln\frac{q(x)}{p(x)}-D_{\rm KL}(\bQ\|\bP)\right|$ and
\begin{align}\label{eq:sigma2-def}
\sigma^2 & =\bE_{\bQ}\left[(U_k-U_{k-1})^2\vert {\cal F}_{k-1}\right] \nonumber
\\
& =\int_{\cal X}q(x)\left(\ln\frac{q(x)}{p(x)}-D_{\rm KL}(\bQ\|\bP)\right)^2\,\rd\mu(x).
\end{align}
In view of a refined Azuma--Hoeffding/McDiarmid inequality \cite{A,McD},
\begin{align}\label{eq:azuma-bound}
{\bP}_{\bQ}\bigg( 
L^*(X_1,\ldots ,X_n)  > (D_{\rm KL}(\bQ\|\bP)+ \beta) 
n-(\ln E_{\varphi}(q)-\ln E_{\varphi}(p))n
\bigg)      \nonumber
\\
   \leq
\exp\left\{-n D\left(\frac{\delta+\gamma^*}{1+\gamma^*}\bigg\|\frac{\gamma^*}{1+\gamma^*}\right)\right\},
\end{align}
where $\delta=\frac{\sigma^2}{d^2}$, $\gamma^*=\frac{\beta^*}{d}$, $\beta^*=\beta-\ln E_{\varphi}(q)+\ln E_{\varphi}(p)$ and
$D(p\|q)=p\ln\frac{p}{q}+(1-p)\ln\frac{1-p}{1-q}$
stands for the Kullback--Leibler divergence between the two Bernoulli distributions $(p,1-p)$ and $(q, 1-q)$.

We use the following modified version of Azuma--Hoeffding inequality; see \cite{McD}.

\begin{Lemma}\label{lem:mcdiarmid}
Let $\{U_k,{\cal F}_k\}$ be a discrete-time real-valued martingale. Assume that, for
some constants $d, \sigma > 0$, the following two requirements are satisfied a.s.\
for every $k \in \{1,\ldots, n\}:$
\begin{equation}\label{eq:mcdiarmid-assumptions}
\begin{array}{l}
| U_k-U_{k-1}| \leq d,
\\
{\rm Var}[U_k-U_{k-1}\vert {\cal F}_{k-1}]\leq \sigma^2.
\end{array}
\end{equation}
Then, for every $\beta \geq 0$,
\begin{equation}\label{eq:mcdiarmid-bound}
{\bP}\left(| U_n-U_0| \geq\beta n\right)\leq
2\exp\left\{-n D\left(\frac{\delta+\gamma}{1+\gamma}\bigg\|\frac{\gamma}{1+\gamma}\right)\right\},
\end{equation}
where $\delta=\frac{\sigma^2}{d^2}$ and $\gamma=\frac{\beta}{d}$.
(Note that $\delta\in[0,1]$ automatically whenever $|U_k-U_{k-1}|\le d$ a.s.)
\end{Lemma}

In particular, the one-sided bound $\bP(U_n-U_0\ge \beta n)\le \exp\{-nD(\cdot)\}$ holds.

\begin{Theorem}\label{thm:Lstar-tail}
Set
\[
\beta^*=\beta-\ln E_{\varphi}(p)+\ln E_{\varphi}(q),
\qquad
\gamma^*=\frac{\beta^*}{d}.
\]
Under conditions $d<\infty$, $\sigma^2<\infty$, and $\beta^*\ge 0$,
\begin{equation}\label{eq:Lstar-tail}
{\bP}_{\bQ}\big(L^*(X_1,\ldots ,X_n)\geq \beta n\big)\leq
\exp\left\{-n D\left(\frac{\delta+\gamma^*}{1+\gamma^*}\bigg\|\frac{\gamma^*}{1+\gamma^*}\right)\right\}.
\end{equation}
\end{Theorem}

Lemma~\ref{lem:mcdiarmid} is quoted from \cite{McD}. Theorem~\ref{thm:Lstar-tail} is its direct application to the
tilted log-likelihood ratio \eqref{eq:Lstar-def}; the only dependence on the context weight $\varphi$ is through
$E_\varphi(p)$ and $E_\varphi(q)$.

 \section{Examples and Applications}\label{sec:examples}

\added{The identities of Section~\ref{sec:main-results} reduce the computation of $D^{\rw}_{B,\alpha}$ and $D^{\rw}_C$ to the single-letter weighted affinity
\[
\rho^{\rw}_{\alpha}(p,q)=\int \varphi(x)\,p(x)^{\alpha}q(x)^{1-\alpha}\,\rd\mu(x),
\]
followed by optimisation over $\alpha\in[0,1]$. We work this out for Gaussian, Poisson, and exponential families, highlighting how the context weight $\varphi$ modifies the classical formulas. When $\varphi\equiv 1$, the expressions reduce to the standard unweighted Bhattacharyya and Chernoff quantities; more involved non-exponential-family computations (such as the Cauchy location--scale family) are deferred to the Appendix.}

{\subsection{A Numerical Illustration}\label{subsec:numerical-illustration}

This subsection illustrates the behaviour of $\alpha^\ast$ and $D_C^{\rw}$ under a non-trivial factorised weight and serves as a direct numerical verification of the Bregman identities of Section~\ref{subsec:bregman}: for the model below the affinity $\rho_\alpha^{\rw}$ is available in closed form, and we check that closed-form evaluation and direct numerical integration agree to machine precision.

Consider the asymmetric Gaussian hypotheses
\[
H_0:\ \bP=\cN(\mu_0,\sigma_0^2),\qquad
H_1:\ \bQ=\cN(\mu_1,\sigma_1^2),\qquad
(\mu_0,\sigma_0^2)=(0,1),\ (\mu_1,\sigma_1^2)=(3,2),
\]
with a non-indicator factorised weight
\begin{equation}\label{eq:illust-weight-gauss}
\varphi(x)\;=\;\exp\bigl(-\beta(x-x_0)^2\bigr),\qquad x_0=0,\ \beta\ge 0.
\end{equation}
At $\beta=0$, one has $\varphi\equiv 1$ and the unweighted Chernoff information is recovered. For $\beta>0$, the weight concentrates near $x_0=\mu_0$; in particular, \eqref{eq:illust-weight-gauss} is not an indicator-type weight, so the weighted problem does not reduce to the unweighted Chernoff information on a restricted domain.

The asymmetry $\sigma_0^2\neq\sigma_1^2$ and the centring of $\varphi$ at the $H_0$ mean are essential for the illustration. In a fully symmetric configuration ($\sigma_0=\sigma_1$, $\mu_1=-\mu_0$, $x_0=0$), the problem is invariant under $\alpha\leftrightarrow 1-\alpha$, so the optimum is pinned to $\alpha^\ast=1/2$ for every $\beta$ and the effect of the weight on the Chernoff compromise is invisible. Asymmetric hypotheses are precisely where the weighted formalism is operationally distinct from the classical one, and it is this distinction that the numerics below is designed to expose.

Writing $A(\alpha)=\alpha/\sigma_0^2+(1-\alpha)/\sigma_1^2+2\beta$, $B(\alpha)=\alpha\mu_0/\sigma_0^2+(1-\alpha)\mu_1/\sigma_1^2+2\beta x_0$, and $C(\alpha)=\alpha\mu_0^2/\sigma_0^2+(1-\alpha)\mu_1^2/\sigma_1^2+2\beta x_0^2$, a direct Gaussian integration yields
\begin{equation}\label{eq:illust-rho-closed}
\ln\rho_\alpha^{\rw}(p,q)
=\tfrac12\ln\!\tfrac{2\pi}{A(\alpha)}
-\tfrac{\alpha}{2}\ln(2\pi\sigma_0^2)
-\tfrac{1-\alpha}{2}\ln(2\pi\sigma_1^2)
-\tfrac12\Bigl(C(\alpha)-\tfrac{B(\alpha)^2}{A(\alpha)}\Bigr),
\end{equation}
and maximising \eqref{eq:illust-rho-closed} over $\alpha\in[0,1]$ gives $\alpha^\ast(\beta)$ and $D_C^{\rw}(\bP,\bQ)(\beta)$. Table~\ref{tab:num-ill} reports their values at three selected $\beta$. Direct numerical integration of $\rho_\alpha^{\rw}$ from its definition agrees with \eqref{eq:illust-rho-closed} to machine precision on all tabulated entries, which confirms the Bregman identities of Section~\ref{subsec:bregman} numerically for this example.

\begin{table}[H]
\caption{Optimal skewing parameter and weighted Chernoff information for the Gaussian example \eqref{eq:illust-weight-gauss}, obtained by maximising \eqref{eq:illust-rho-closed}.\label{tab:num-ill}}
\begin{tabularx}{\textwidth}{rCC}
\toprule
\boldmath{$\beta$} & \boldmath{$\alpha^\ast(\beta)$} & \textbf{\boldmath{$D_C^{\rw}(\bP,\bQ)$}} \\
\midrule
$0$ (unweighted) & $0.4153$ & $0.8018$ \\
$1/16$           & $0.3355$ & $0.9827$ \\
$1/4$            & $0.0963$ & $1.4935$ \\
\bottomrule
\end{tabularx}
\end{table}

The monotone growth of $\beta\mapsto D_C^{\rw}$ and the leftward shift of $\alpha^\ast(\beta)$ illustrate a qualitative conclusion of Section~\ref{sec:main-results}: localising $\varphi$ near $\mu_0$ preferentially retains observations that are more probable under $H_0$ and thereby increases the effective discrimination rate, while simultaneously moving the optimal tilting towards the $H_0$ side. The classical unweighted limit is recovered at $\beta=0$.

In the language of hypothesis testing, $\alpha^\ast$ is the parameter that balances the exponential rates of the type-I and type-II losses at the Bayes optimum: the type-I exponent equals $\alpha^\ast D_C^{\rw}$ and the type-II exponent equals $(1-\alpha^\ast)D_C^{\rw}$ (cf.\ Section~\ref{sec:main-results}). A leftward shift of $\alpha^\ast$ therefore corresponds to reallocating the available exponential budget towards faster decay of the type-II loss at the expense of the type-I loss, which is the optimal response to a weight that concentrates mass in the region where $H_0$ is more plausible.

\subsubsection*{Data and Code Availability}
A Jupyter/Colab notebook reproducing Table~\ref{tab:num-ill} and Figures~\ref{fig:num-ill-alpha}--\ref{fig:num-ill-DC}, together with the direct-integration verification of \eqref{eq:illust-rho-closed}, is archived on Zenodo~\cite{Kalimulina2026Code} and mirrored on GitHub.
}
\begin{center}
\begin{figure}[H]
\includegraphics[width=0.9\linewidth]{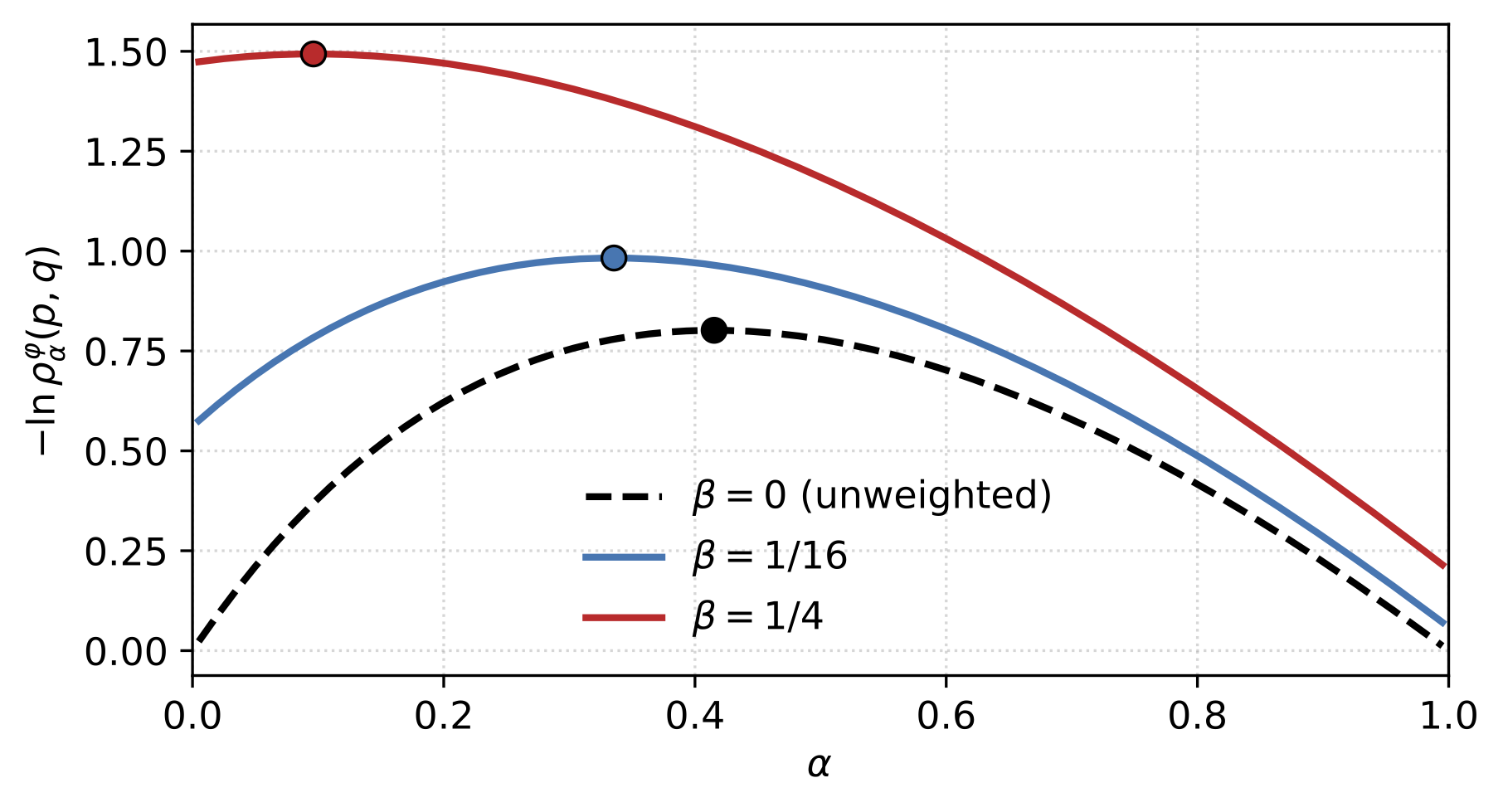}
\caption{The map $\alpha\mapsto-\ln\rho_\alpha^{\rw}(p,q)$ for the Gaussian hypotheses $\cN(0,1)$, $\cN(3,2)$ with weight \eqref{eq:illust-weight-gauss}, for $\beta\in\{0,1/16,1/4\}$. The optimum $\alpha^\ast$ is marked by a bullet on each curve and shifts to the left as $\beta$ increases.}
\label{fig:num-ill-alpha}
\end{figure}
\end{center}
 \vspace{-12pt}
 \begin{center}
\begin{figure}[H]
\includegraphics[width=0.9\linewidth]{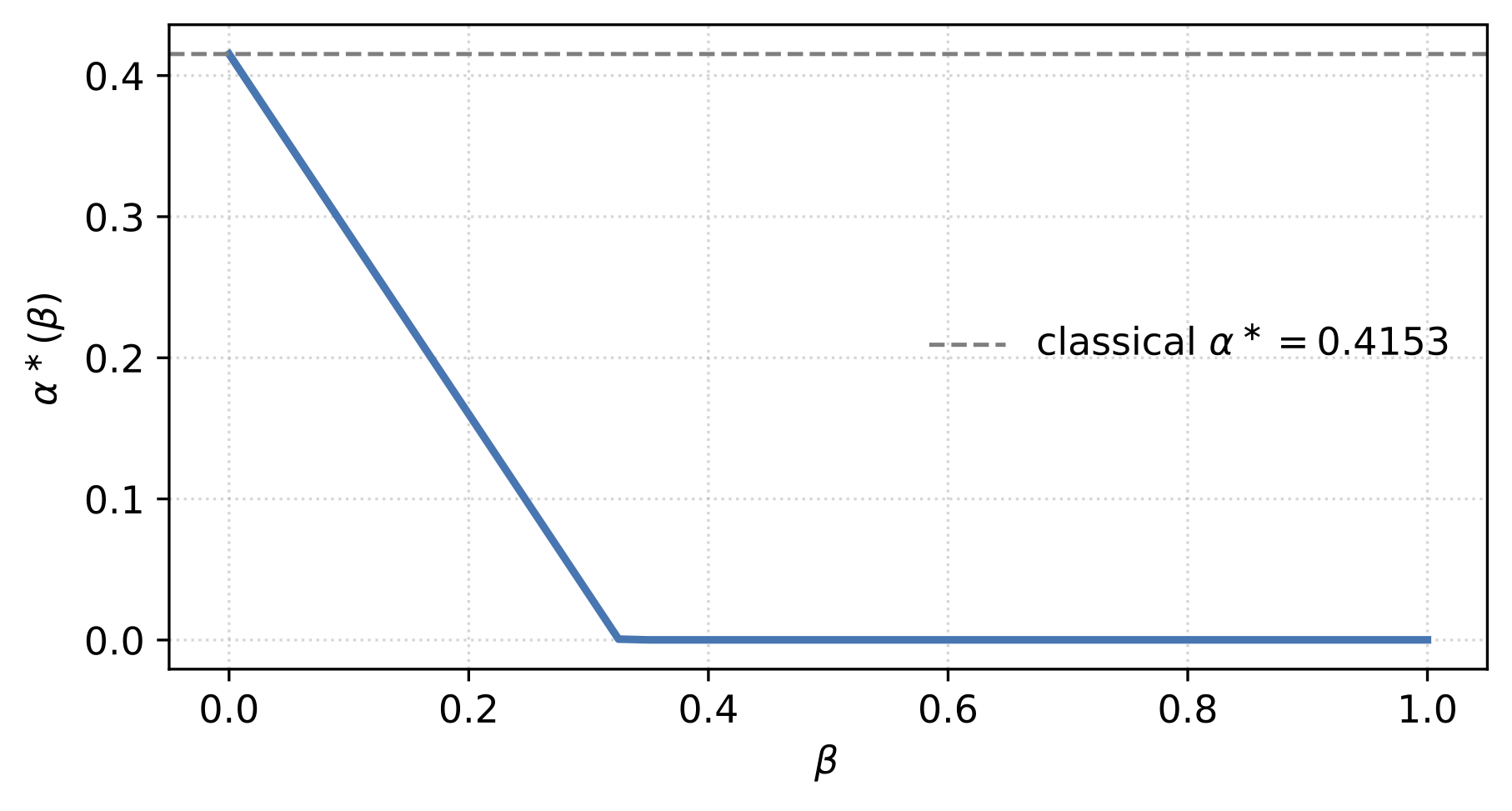}
\caption{Optimal skewing parameter $\alpha^\ast(\beta)$ for the Gaussian example. The dashed line marks the unweighted value $\alpha^\ast(0)$.}
\label{fig:num-ill-alpha-star}
\end{figure}
\end{center}
\vspace{-12pt}
\begin{center}
\begin{figure}[H]
\includegraphics[width=0.9\linewidth]{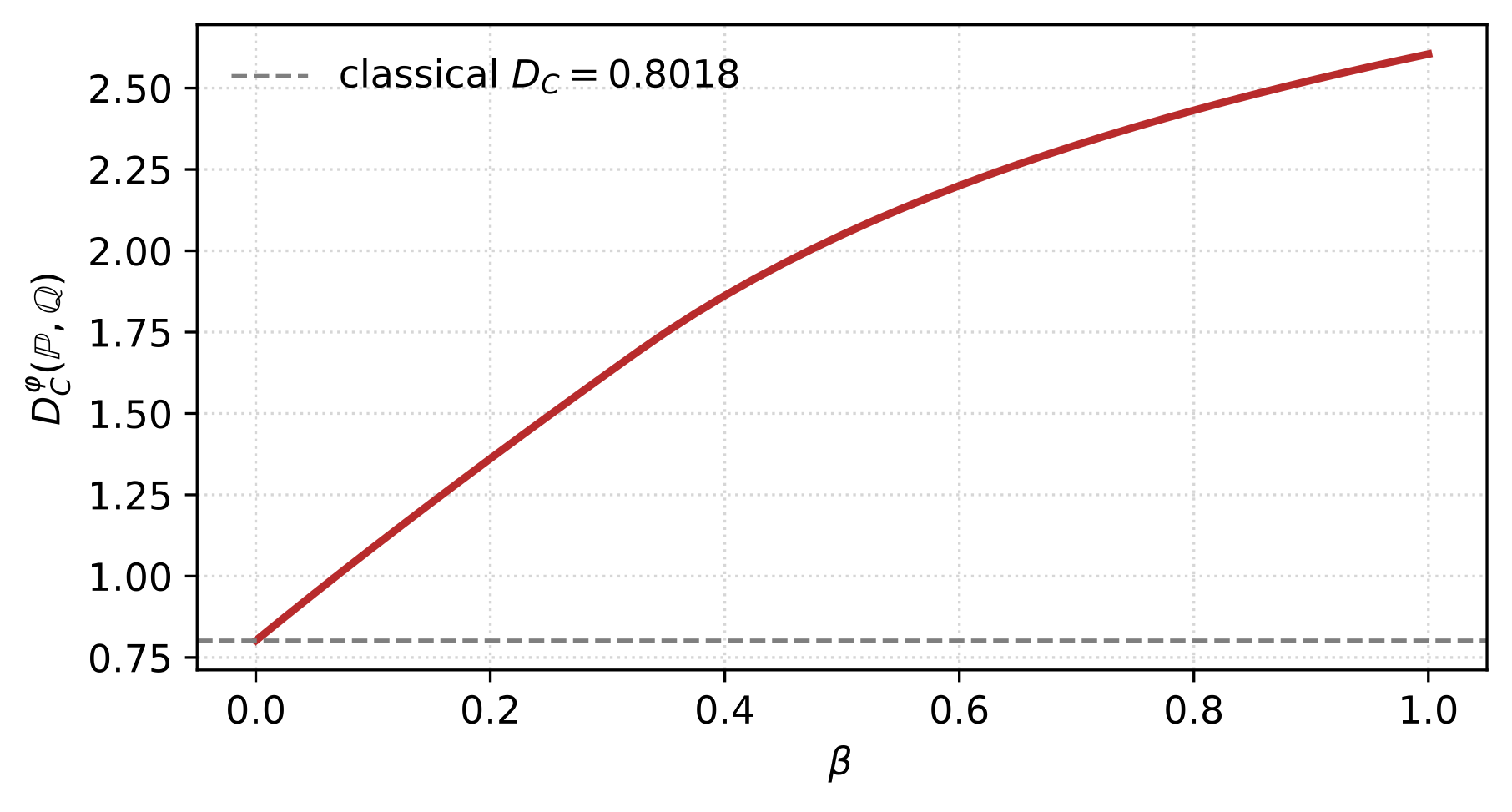}
\caption{Weighted Chernoff information $\beta\mapsto D_C^{\rw}(\bP,\bQ)$. The dashed line marks the classical value $D_C$ recovered at $\beta=0$.}
\label{fig:num-ill-DC}
\end{figure}
\end{center}
\subsection{Gaussian Models}\label{subsec:gauss-examples}

Throughout this subsection, the reference measure is the Lebesgue measure on $\bR^d$.
We compute the weighted Bhattacharyya coefficient
\[
\rho^{\rw}_{\alpha}(\bP,\bQ)=\int_{\bR^d}\varphi(x)\,p(x)^{\alpha}q(x)^{1-\alpha}\,\rd x,
\qquad \alpha\in[0,1],
\]
together with the weighted Bhattacharyya distance
$D^{\rw}_{B,\alpha}(\bP,\bQ):=-\ln\rho^{\rw}_{\alpha}(\bP,\bQ)$ and the weighted Chernoff information
$D^{\rw}_{C}(\bP,\bQ)=\max_{\alpha\in[0,1]}D^{\rw}_{B,\alpha}(\bP,\bQ)$ (Definition~\ref{def:weighted-chernoff}).
Note that, unlike the unweighted case, $\rho^{\rw}_{\alpha}$ is not restricted to $[0,1]$ and
$D^{\rw}_{B,\alpha}$ (hence also $D_C^{\rw}$) may take negative values.

\begin{example}[Gaussian weighted Bhattacharyya coefficient with exponential weight]\label{ex:gauss-expweight}
Let $\bP=\mathcal N(\mu_1,\Sigma_1)$ and $\bQ=\mathcal N(\mu_2,\Sigma_2)$ on $\bR^d$, where
$\Sigma_1\succ0$ and $\Sigma_2\succ0$, and let $\varphi(x)=e^{\gamma^{\rT}x}$ for some $\gamma\in\bR^d$.
Denote by $p,q$ the corresponding densities.
For $\alpha\in[0,1]$ define
\begin{align*}
\Sigma_{\alpha}&:=\left(\alpha\Sigma_1^{-1}+(1-\alpha)\Sigma_2^{-1}\right)^{-1},
\\
\widetilde\mu_{\alpha}
&:=\Sigma_{\alpha}\Big(\alpha\Sigma_1^{-1}\mu_1+(1-\alpha)\Sigma_2^{-1}\mu_2+\gamma\Big).
\end{align*}
Then
\begin{align}\label{eq:gauss-rho-expweight}
&\rho^{\rw}_{\alpha}(\bP,\bQ)
=\int_{\bR^d}e^{\gamma^{\rT}x}\,p(x)^{\alpha}q(x)^{1-\alpha}\,\rd x
\nonumber\\
&=\frac{|\Sigma_{\alpha}|^{1/2}}{|\Sigma_1|^{\alpha/2}|\Sigma_2|^{(1-\alpha)/2}}
\exp\!\left\{-\frac12\Big(\alpha\mu_1^{\rT}\Sigma_1^{-1}\mu_1+(1-\alpha)\mu_2^{\rT}\Sigma_2^{-1}\mu_2
-\widetilde\mu_{\alpha}^{\rT}\Sigma_{\alpha}^{-1}\widetilde\mu_{\alpha}\Big)\right\}.
\end{align}
Consequently,
\begin{align}\label{eq:gauss-Db-expweight}
&D^{\rw}_{B,\alpha}(\bP,\bQ)=  \nonumber
\\
&=\frac12\left(
\alpha\mu_1^{\rT}\Sigma_1^{-1}\mu_1+(1-\alpha)\mu_2^{\rT}\Sigma_2^{-1}\mu_2
-\widetilde\mu_{\alpha}^{\rT}\Sigma_{\alpha}^{-1}\widetilde\mu_{\alpha}
+\ln\frac{|\Sigma_1|^{\alpha}|\Sigma_2|^{1-\alpha}}{|\Sigma_{\alpha}|}
\right).
\end{align}
In particular, setting $\gamma=0$ (i.e.,\ $\varphi\equiv 1$) reduces \eqref{eq:gauss-Db-expweight}
to the classical (unweighted) Gaussian Bhattacharyya distance; see, e.g., \cite{N}.
\end{example}

\begin{Corollary}[Common covariance]\label{cor:gauss-commonSigma-expweight}
In Example~\ref{ex:gauss-expweight}, assume $\Sigma_1=\Sigma_2=\Sigma\succ0$ and keep the exponential weight
$\varphi(x)=e^{\gamma^{\rT}x}$.
Set $\delta:=\mu_1-\mu_2$, $\|v\|_{\Sigma^{-1}}^2:=v^{\rT}\Sigma^{-1}v$, and
$\mu_{\alpha}:=\alpha\mu_1+(1-\alpha)\mu_2$.
Then, for any $\alpha\in[0,1]$,
\begin{align}\label{eq:gauss-commonSigma-expweight-rho}
\rho^{\rw}_{\alpha}(\bP,\bQ)
&=\int_{\bR^d} e^{\gamma^{\rT}x}\,p(x)^\alpha q(x)^{1-\alpha}\,\rd x    \nonumber
\\
&=\exp\!\left\{-\frac{\alpha(1-\alpha)}{2}\|\delta\|_{\Sigma^{-1}}^2
+\gamma^{\rT}\mu_\alpha+\frac12\gamma^{\rT}\Sigma\gamma\right\},
\end{align}
and therefore
\begin{equation}\label{eq:gauss-commonSigma-expweight-Db}
D^{\rw}_{B,\alpha}(\bP,\bQ)
=-\ln \rho^{\rw}_{\alpha}(\bP,\bQ)
=\frac{\alpha(1-\alpha)}{2}\|\delta\|_{\Sigma^{-1}}^2
-\gamma^{\rT}\mu_{\alpha}-\frac12\gamma^{\rT}\Sigma\gamma.
\end{equation}
If $\mu_1\neq\mu_2$ and the unconstrained maximiser
\[
\widetilde\alpha=\frac12-\frac{\gamma^{\rT}\delta}{\|\delta\|_{\Sigma^{-1}}^2}
\]
belongs to $(0,1)$, then $\alpha^*=\widetilde\alpha$; otherwise the maximum over $\alpha\in[0,1]$
is attained at the nearest boundary point $\alpha^*\in\{0,1\}$.
In all cases,
\[
D_C^{\rw}(\bP,\bQ)=\max_{\alpha\in[0,1]}D^{\rw}_{B,\alpha}(\bP,\bQ)=D^{\rw}_{B,\alpha^*}(\bP,\bQ).
\]
In particular, for $\gamma=0$ (i.e.,\ $\varphi\equiv 1$) we recover $\alpha^*=1/2$ and
$D_C(\bP,\bQ)=\|\delta\|_{\Sigma^{-1}}^2/8$.
\end{Corollary}

\begin{proof}
The expression for $\rho^{\rw}_\alpha$ follows by simplifying Example~\ref{ex:gauss-expweight} under
$\Sigma_1=\Sigma_2=\Sigma$, which makes the determinant prefactor equal to $1$ and yields a Gaussian MGF term
$\exp(\gamma^{\rT}\mu_\alpha+\tfrac12\gamma^{\rT}\Sigma\gamma)$.
The maximiser follows by differentiating \eqref{eq:gauss-commonSigma-expweight-Db} in $\alpha$.
\end{proof}


Choosing an exponential weight $\varphi(x)=e^{\gamma^{\rT}x}$ corresponds to exponential tilting: for a Gaussian $X\sim\mathcal N(\mu,\Sigma)$, this tilting keeps the covariance and shifts the mean to $\mu+\Sigma\gamma$ (with normalisation factor $\exp(\gamma^{\rT}\mu+\tfrac12\gamma^{\rT}\Sigma\gamma)$), which is why the weighted affinities remain available in closed form. 
In particular, the optimal Chernoff parameter is no longer forced to be $\alpha^*=1/2$ and, as \eqref{eq:gauss-commonSigma-expweight-Db} shows, sufficiently strong tilting can push the maximiser to the boundary $\alpha^*\in\{0,1\}$.

\subsection{Poisson Models}\label{subsec:poisson-examples}

\begin{example}[Poisson model with exponential weight]\label{ex:poisson}
Let ${\cal X}=\mathbb N_0=\{0,1,2,\dots\}$ and let $\mu$ be the counting measure on ${\cal X}$.
Fix two hypotheses $\bP=\mathrm{Poi}(\lambda_1)$ and $\bQ=\mathrm{Poi}(\lambda_2)$ with $\lambda_1,\lambda_2>0$,
and write $p=p_{\lambda_1}$ and $q=q_{\lambda_2}$.
{Throughout this subsection we work under the standing assumption of Section~\ref{sec:setup}, namely, that the observations $X_1,\ldots,X_n$ are i.i.d.\ (distributed as $\bP$ under $H_0$ and as $\bQ$ under $H_1$), and that the weight $\varphi$ factorises across observations (Assumption~\ref{ass:phi-factorised}).}
We still consider the exponential weight
\(
\varphi_\gamma(k)=e^{\gamma k},~ \gamma\in\mathbb R.
\)
(For $\gamma=0$ we recover the unweighted case $\varphi\equiv 1$.)
{Equivalently, setting $\varepsilon:=e^\gamma>0$, the weight takes the form $\varphi(k)=\varepsilon^k$; this reparameterisation is convenient in applications where $\varepsilon\in(0,1)$ models a per-event discount factor, while $\gamma\in\mathbb R$ is the natural parameter for the exponential-family calculations below. The two parameterisations are equivalent.}

For $\alpha\in[0,1]$, set
\[
\lambda_\alpha:=\lambda_1^\alpha\lambda_2^{\,1-\alpha}.
\]

\smallskip
\noindent\textbf{(a) Weighted Bhattacharyya coefficient and Chernoff arc. }
A direct summation gives
\begin{align}\label{eq:poisson-rho}
\rho^{\rw}_\alpha(\bP,\bQ)
&=\sum_{k=0}^\infty \varphi_\gamma(k)\,p(k)^\alpha q(k)^{1-\alpha}
=\exp\!\big\{-\alpha\lambda_1-(1-\alpha)\lambda_2+e^\gamma\lambda_\alpha\big\}.
\end{align}
Hence, by Definition~\ref{def:weighted-bhat},
\begin{equation}\label{eq:poisson-Db}
D^{\rw}_{B,\alpha}(\bP,\bQ)
=-\ln\rho^{\rw}_\alpha(\bP,\bQ)
=\alpha\lambda_1+(1-\alpha)\lambda_2-e^\gamma\lambda_\alpha.
\end{equation}

Moreover, the normalised weighted geometric mixture (Chernoff-tilted density) from \eqref{eq:Epq-def}
takes the form
\[
(pq)_\alpha(k)=\frac{\varphi_\gamma(k)\,p(k)^\alpha q(k)^{1-\alpha}}{\rho^{\rw}_\alpha(\bP,\bQ)}
=\exp\!\{-e^\gamma\lambda_\alpha\}\,\frac{(e^\gamma\lambda_\alpha)^k}{k!}
=\mathrm{Poi}(e^\gamma\lambda_\alpha).
\]

\smallskip
\noindent\textbf{(b) Optimal Chernoff parameter.}
If $\lambda_1\neq\lambda_2$, then $\alpha\mapsto D^{\rw}_{B,\alpha}(\bP,\bQ)$ is strictly concave on $[0,1]$ since
\[
\frac{\rd^2}{\rd\alpha^2}D^{\rw}_{B,\alpha}(\bP,\bQ)
=-e^\gamma\lambda_\alpha\left(\ln\frac{\lambda_1}{\lambda_2}\right)^2<0;
\]
hence, the maximiser $\alpha^*$ in Definition~\ref{def:weighted-chernoff} is unique.
Differentiating \eqref{eq:poisson-Db} yields the critical point condition
\[
0=\frac{\rd}{\rd\alpha}D^{\rw}_{B,\alpha}(\bP,\bQ)
=\lambda_1-\lambda_2-e^\gamma\lambda_\alpha\ln\!\Big(\frac{\lambda_1}{\lambda_2}\Big).
\]

Equivalently, the (unconstrained) critical point $\alpha=\widetilde\alpha$ satisfies
\[
\lambda_{\widetilde\alpha}=e^{-\gamma}L(\lambda_1,\lambda_2),
\qquad
L(\lambda_1,\lambda_2):=\frac{\lambda_1-\lambda_2}{\ln\lambda_1-\ln\lambda_2}.
\]

In contrast to the unweighted case ($\gamma=0$), the context tilt $\gamma$ may push the optimal
Chernoff parameter to the boundary $\alpha^*\in\{0,1\}$.

Thus, the unconstrained maximiser is
\[
\widetilde\alpha
=\frac{\ln L(\lambda_1,\lambda_2)-\gamma-\ln\lambda_2}{\ln\lambda_1-\ln\lambda_2},
\]
and the maximiser on $[0,1]$ is $\alpha^*=\Pi_{[0,1]}(\widetilde\alpha)$, where
\[
\Pi_{[0,1]}(a):=\min\{1,\max\{0,a\}\}.
\]
Finally,
\[
D_C^{\rw}(\bP,\bQ)=\max_{\alpha\in[0,1]}D^{\rw}_{B,\alpha}(\bP,\bQ)=D^{\rw}_{B,\alpha^*}(\bP,\bQ).
\]

If $\lambda_1=\lambda_2$, then $\rho^{\rw}_\alpha(\bP,\bQ)$ does not depend on $\alpha$ and
$D_C^{\rw}(\bP,\bQ)=D^{\rw}_{B,\alpha}(\bP,\bQ)$ for any $\alpha\in[0,1]$.
\end{example}

\noindent\textit{Derivation of \eqref{eq:poisson-rho}.}
For $k\in\mathbb N_0$,
\[
p(k)^\alpha q(k)^{1-\alpha}
=\exp\!\{-\alpha\lambda_1-(1-\alpha)\lambda_2\}\,\frac{\lambda_\alpha^k}{k!},
\]
so multiplying by $e^{\gamma k}$ and summing over $k$ gives
\begin{align*}
    \rho^{\rw}_\alpha(\bP,\bQ)  &   =\exp\{-\alpha\lambda_1-(1-\alpha)\lambda_2\}\sum_{k\ge0}(e^\gamma\lambda_\alpha)^k/k!
    \\
& =\exp\{-\alpha\lambda_1-(1-\alpha)\lambda_2+e^\gamma\lambda_\alpha\}.
\end{align*} \hfill$\Box$

\subsection{Exponential Models}\label{subsec:exp-examples}

\begin{example}[Exponential model with exponential weight]\label{ex:exp}
Let ${\cal X}=\bR_+=[0,\infty)$ with Lebesgue measure.
Fix two hypotheses $\bP=\mathrm{Exp}(\lambda_1)$ and $\bQ=\mathrm{Exp}(\lambda_2)$ with rates
$\lambda_1,\lambda_2>0$, and write
\[
p(x)=\lambda_1 e^{-\lambda_1 x}{\bf 1}\{x\ge 0\},\qquad
q(x)=\lambda_2 e^{-\lambda_2 x}{\bf 1}\{x\ge 0\}.
\]
Consider the exponential weight $\varphi_\gamma(x)=e^{\gamma x}$ with
\[
\gamma<\min\{\lambda_1,\lambda_2\},
\]
so that $\rho^{\rw}_\alpha(\bP,\bQ)\in(0,\infty)$ for all $\alpha\in[0,1]$.
For $\alpha\in[0,1]$ set
\[
\lambda_\alpha:=\alpha\lambda_1+(1-\alpha)\lambda_2 .
\]

\smallskip
\noindent\textbf{(a) Weighted Bhattacharyya coefficient and Chernoff arc.}
A direct computation gives
\[
\rho^{\rw}_{\alpha}(\bP,\bQ)
=\int_{0}^{\infty} e^{\gamma x}\,p(x)^{\alpha}q(x)^{1-\alpha}\,\rd x
=\frac{\lambda_1^{\alpha}\lambda_2^{\,1-\alpha}}{\lambda_\alpha-\gamma}.
\]
Hence
\[
D^{\rw}_{B,\alpha}(\bP,\bQ)
=-\ln\rho^{\rw}_{\alpha}(\bP,\bQ)
=\ln(\lambda_\alpha-\gamma)-\alpha\ln\lambda_1-(1-\alpha)\ln\lambda_2.
\]
Moreover, the Chernoff-tilted density $(pq)_\alpha$ from \eqref{eq:Epq-def} is again exponential:
\[
(pq)_\alpha(x)
=\frac{e^{\gamma x}p(x)^{\alpha}q(x)^{1-\alpha}}{\rho^{\rw}_{\alpha}(\bP,\bQ)}
=(\lambda_\alpha-\gamma)e^{-(\lambda_\alpha-\gamma)x}{\bf 1}\{x\ge 0\}
=\mathrm{Exp}(\lambda_\alpha-\gamma).
\]

\smallskip
\noindent\textbf{(b) Optimal Chernoff parameter.}
If $\lambda_1\neq\lambda_2$, then $\alpha\mapsto D^{\rw}_{B,\alpha}(\bP,\bQ)$ is strictly concave on $[0,1]$;
hence, the maximiser $\alpha^*$ in Definition~\ref{def:weighted-chernoff} is unique.
Differentiating yields the critical point condition
\[
\frac{\lambda_1-\lambda_2}{\lambda_\alpha-\gamma}=\ln\!\Big(\frac{\lambda_1}{\lambda_2}\Big).
\]
Equivalently, the (unconstrained) critical point $\alpha=\widetilde\alpha$ satisfies
\[
\lambda_{\widetilde\alpha}-\gamma=L(\lambda_1,\lambda_2),
\qquad
L(\lambda_1,\lambda_2):=\frac{\lambda_1-\lambda_2}{\ln\lambda_1-\ln\lambda_2},
\]
so that
\[
\widetilde\alpha=\frac{\gamma+L(\lambda_1,\lambda_2)-\lambda_2}{\lambda_1-\lambda_2}.
\]
The maximiser on $[0,1]$ is $\alpha^*=\Pi_{[0,1]}(\widetilde\alpha)$ (projection onto $[0,1]$), and
\[
D_C^{\rw}(\bP,\bQ)=\max_{\alpha\in[0,1]}D^{\rw}_{B,\alpha}(\bP,\bQ)=D^{\rw}_{B,\alpha^*}(\bP,\bQ).
\]

If $\lambda_1=\lambda_2=\lambda$, then $\rho^{\rw}_\alpha(\bP,\bQ)=\lambda/(\lambda-\gamma)$ does not depend on $\alpha$,
so any $\alpha\in[0,1]$ is optimal and $D_C^{\rw}(\bP,\bQ)=\ln(\lambda-\gamma)-\ln\lambda$.
Setting $\gamma=0$ (i.e.,\ $\varphi\equiv 1$) recovers the classical unweighted expressions.
\end{example}

\subsubsection*{Additional Example (Baseline, Non-Exponential Family)}
Appendix~\ref{app:cauchy} contains a closed-form illustration for the Cauchy location--scale family.
Since the Cauchy family is not an exponential family, this example complements the main text by showing that,
even in the unweighted baseline case $\varphi\equiv 1$, the Bhattacharyya coefficient (in particular $\rho_{1/2}$)
and the Chernoff information may involve special functions (complete elliptic integrals).
For nontrivial weights $\varphi$, the symmetry $\rho_\alpha=\rho_{1-\alpha}$ (hence $\alpha^*=1/2$) typically fails and
a comparable closed form is not available, so we keep the baseline Cauchy computation in the appendix.

\subsection{Extension to $M$-ary Hypothesis Testing}\label{subsec:mary}

\added{We now record the finite-$M$ analogue of Theorem~\ref{thm:optimal-sum-loss}. The key observation is that the optimal $M$-ary pointwise loss is squeezed between pairwise minima (Lemma~\ref{lem:mary-PairwiseMinima}), and each pairwise term has logarithmic rate given by the corresponding weighted Chernoff information. Hence the overall $M$-ary rate is determined by the closest pair in terms of $D_C^{\rw}$.}

Fix an integer $M\ge 2$ and let $\bP_1,\ldots,\bP_M$ be probability measures on ${\cal X}$
dominated by $\mu$, with strictly positive densities $p_1,\ldots,p_M$.
Let $X_1^n=(X_1,\ldots,X_n)$ be i.i.d.\ under each hypothesis
$H_i:\ X_1^n\sim \bP_i^{\otimes n}$.
Assume that the weight function factorises as in Assumption~\ref{ass:phi-factorised}.

Assume moreover that for every $1\le i<j\le M$ and every $\alpha\in[0,1]$,
\[
\rho^{\rw}_\alpha(p_i,p_j)
=\int_{\cal X}\varphi(x)\,p_i(x)^\alpha p_j(x)^{1-\alpha}\,\rd\mu(x)\in(0,\infty),
\]
\added{so that all pairwise weighted Chernoff information values are well-defined
and inequality
\begin{equation}\label{finite}
\max\limits_{i\not=j}\sup\limits_{\alpha\in [0,1]}\int\varphi(x)|\ln\frac{p_i(x)}{p_j(x)}| p_i(x)^{\alpha} p_j(x)^{1-\alpha}
{\rd}\mu(x)<\infty
\end{equation}
holds true.}

A (deterministic) $M$-ary decision rule is a measurable map $\delta_n:{\cal X}^n\to\{1,\ldots,M\}$.
Define the context-sensitive loss under $H_i$ by
\[
L_{i,n}(\delta_n):=\bE_{\bP_i^{\otimes n}}\!\left[\varphi(X_1^n)\,{\bf 1}\{\delta_n(X_1^n)\neq i\}\right],
\]
and the total loss
\[
L_{n,M}(\delta_n):=\sum_{i=1}^M L_{i,n}(\delta_n),\qquad
L_{n,M}^*:=\inf_{\delta_n} L_{n,M}(\delta_n).
\]

\begin{Proposition}[Pointwise form of the optimal $M$-ary total loss]\label{prop:mary-pointwise}
For each $n\ge 1$,
\begin{equation}\label{eq:mary-opt-integral}
L_{n,M}^*=\int_{{\cal X}^n}\varphi(x_1^n)\left(\sum_{i=1}^M p_i(x_1^n)-\max_{1\le j\le M}p_j(x_1^n)\right)\,\rd\mu^{\otimes n}(x_1^n),
\end{equation}
where $p_i(x_1^n)=\prod_{k=1}^n p_i(x_k)$. Moreover, an optimal rule is given by the maximum-likelihood
classifier
\[
\delta_n^*(x_1^n)\in\arg\max_{1\le j\le M} p_j(x_1^n)
\]
(with any measurable tie-breaking).
\end{Proposition}

\begin{proof}
Fix $\delta_n$. Using $\sum_{i=1}^M {\bf 1}\{\delta_n\neq i\}p_i
= \sum_{i=1}^M p_i - p_{\delta_n}$ pointwise, we obtain
\[
L_{n,M}(\delta_n)
=\int_{{\cal X}^n} \varphi(x_1^n)\left(\sum_{i=1}^M p_i(x_1^n)-p_{\delta_n(x_1^n)}(x_1^n)\right)\rd\mu^{\otimes n}(x_1^n).
\]
Minimisation over $\delta_n$ is therefore pointwise in $x_1^n$ and is achieved by selecting an index
maximising $p_j(x_1^n)$, yielding \eqref{eq:mary-opt-integral}.
\end{proof}

\begin{Lemma}[Pairwise minima]\label{lem:mary-PairwiseMinima}
For any non-negative numbers $a_1,\ldots,a_M$,
\begin{equation}\label{eq:PairwiseMinima}
\max_{1\le i<j\le M}\min(a_i,a_j)\ \le\ \sum_{k=1}^M a_k-\max_{1\le k\le M}a_k\ \le\ \sum_{1\le i<j\le M}\min(a_i,a_j).
\end{equation}
Consequently, defining for $i<j$
\[
I_n^{i,j}:=\int_{{\cal X}^n}\varphi(x_1^n)\min\{p_i(x_1^n),p_j(x_1^n)\}\,\rd\mu^{\otimes n}(x_1^n),
\]
we have the sandwich inequality
\begin{equation}\label{eq:mary-risk-sandwich}
\max_{i<j} I_n^{i,j}\ \le\ L_{n,M}^*\ \le\ \sum_{i<j} I_n^{i,j}.
\end{equation}
\end{Lemma}

\begin{proof}
Let $a_{(1)}\ge a_{(2)}\ge\cdots\ge a_{(M)}$ be the decreasing rearrangement of $(a_1,\ldots,a_M)$.
Then $\sum_k a_k-\max_k a_k=\sum_{r=2}^M a_{(r)}\ge a_{(2)}$.
Moreover, $\max\limits_{i<j}\min(a_i,a_j)=a_{(2)}$, proving the left inequality in \eqref{eq:PairwiseMinima}.

For the right inequality, let $k^*\in\arg\max_k a_k$. Then
\[
\sum_{1\le i<j\le M}\min(a_i,a_j)\ \ge\ \sum_{k\neq k^*}\min(a_k,a_{k^*})
=\sum_{k\neq k^*} a_k=\sum_{k=1}^M a_k-\max_k a_k.
\]
Applying \eqref{eq:PairwiseMinima} pointwise to $a_i=p_i(x_1^n)$, multiplying by $\varphi(x_1^n)$
and integrating yields \eqref{eq:mary-risk-sandwich}.
\end{proof}

\begin{Theorem}[$M$-ary exponent equals the minimum pairwise weighted Chernoff information]\label{thm:mary-exponent}
For $1\le i<j\le M$, let $D_C^{\rw}(\bP_i,\bP_j)$ be the weighted Chernoff information
as in Definition~\ref{def:weighted-chernoff}, and (\ref{finite}) holds true. Set
\[
C_M^{\rw}:=\min_{1\le i<j\le M} D_C^{\rw}(\bP_i,\bP_j).
\]
Then the optimal $M$-ary total loss satisfies
\begin{equation}\label{eq:mary-asymptotic}
L_{n,M}^*=\exp\{-nC_M^{\rw}+o(n)\},\qquad n\to\infty,
\end{equation}
or equivalently,
\[
\lim_{n\to\infty}-\frac{1}{n}\ln L_{n,M}^* = C_M^{\rw}.
\]
\end{Theorem}

\begin{proof}
Fix $1\le i<j\le M$ and consider the binary testing problem between $\bP_i$ and $\bP_j$ with the same
factorised weight $\varphi$. The optimal binary total loss equals
\[
I_n^{i,j}=\int_{{\cal X}^n}\varphi(x_1^n)\min\{p_i(x_1^n),p_j(x_1^n)\}\,\rd\mu^{\otimes n}(x_1^n),
\]
and by Theorem~\ref{thm:optimal-sum-loss} applied to the pair $(\bP_i,\bP_j)$,
\[
I_n^{i,j}=\exp\{-nD_C^{\rw}(\bP_i,\bP_j)+o_{i,j}(n)\}.
\]
Since the number of pairs is finite, letting $r_n:=\max_{i<j}|o_{i,j}(n)|$ yields $r_n=o(n)$ and
\[
I_n^{i,j}=\exp\{-nD_C^{\rw}(\bP_i,\bP_j)+O(r_n)\}\qquad\text{uniformly over }i<j.
\]

Now use the sandwich inequality\eqref{eq:mary-risk-sandwich}. Let $(i^*,j^*)$ attain the minimum
$C_M^{\rw}=D_C^{\rw}(\bP_{i^*},\bP_{j^*})$. From the lower bound,
\[
L_{n,M}^*\ge I_n^{i^*,j^*}=\exp\{-nC_M^{\rw}+O(r_n)\}.
\]
From the upper bound,
\[
L_{n,M}^*\le \sum_{i<j} I_n^{i,j}
\le \binom{M}{2}\exp\{-nC_M^{\rw}+O(r_n)\}.
\]
Taking $-\frac{1}{n}\ln(\cdot)$ and letting $n\to\infty$ yields \eqref{eq:mary-asymptotic}.
\end{proof}

\begin{remark}[Nonzero priors do not change the exponent]\label{rem:mary-priors}
Let $w_1,\ldots,w_M>0$, $\sum_i w_i=1$, and consider the Bayesian weighted total loss
$L_{n,M}^{(w)}(\delta_n)=\sum\limits_{i=1}^M w_i L_{i,n}(\delta_n)$ with optimum
$L_{n,M}^{(w)*}:=\inf_{\delta_n} L_{n,M}^{(w)}(\delta_n)$.
Then, the exponent remains $C_M^{\rw}$.
Indeed, writing $w_{\min}:=\min_i w_i$ and $w_{\max}:=\max_i w_i$, for any $\delta_n$,
\[
w_{\min} L_{n,M}(\delta_n)\le L_{n,M}^{(w)}(\delta_n)\le w_{\max} L_{n,M}(\delta_n),
\]
and taking infimum over $\delta_n$ gives
$w_{\min}L_{n,M}^*\le L_{n,M}^{(w)*}\le w_{\max}L_{n,M}^*$.
Hence, $-\frac1n\ln L_{n,M}^{(w)*}$ and $-\frac1n\ln L_{n,M}^*$ have the same limit $C_M^{\rw}$.
\end{remark}

\section{Conclusions}\label{sec:conclusion}
\added{We studied context-sensitive simple hypothesis testing under a multiplicative weight and proved that the optimal total loss admits the single-letter logarithmic asymptotic
\[
L_n^*=\exp\{-nD_C^{\rw}(\bP,\bQ)+o(n)\}.
\]
The rate is the weighted Chernoff information. The main structural ingredient is an exponential-family embedding of the weighted geometric mixtures $\varphi\,p^\alpha q^{1-\alpha}$, which yields the characterisation of the optimal Chernoff parameter through the log-normaliser and leads to weighted information-geometric identities. We also derived finite-$n$ concentration bounds for the tilted weighted log-likelihood, obtained explicit formulas in Gaussian, Poisson, and exponential models, and extended the logarithmic asymptotic to finitely many hypotheses through the minimum pairwise weighted Chernoff information.}

{\subsection*{Open Problems}%
The single-letter representation \eqref{eq:intro-main-asymptotic} rests on Assumption~\ref{ass:phi-factorised} (factorised weights), on the integrability of the tilted log-normaliser $\hat F$, and on i.i.d.\ sampling of simple hypotheses. Relaxing these assumptions defines several natural open directions.
\begin{itemize}\itemsep2pt
\item[(a)] \emph{Non-factorised weights.} Replacing Assumption~\ref{ass:phi-factorised} by a weight $\varphi(x_1^n)=\psi\bigl(\tfrac{1}{n}\sum_{i=1}^n h(x_i)\bigr)$ or a pairwise-interaction weight; the single-letter rate is then expected to be replaced by a variational formula over the space of probability measures, in the spirit of Sanov/Gibbs conditioning.
\item[(b)] \emph{Integrability.} Weights with heavy tails in the sufficient statistic $t(x)$ may violate the finiteness of $\hat F$ near $\alpha^\ast\theta_1+(1-\alpha^\ast)\theta_2$; the boundary cases $\alpha^\ast\in\{0,1\}$ (cf.\ Proposition~\ref{prop:chernoff-jensen-bisector} and \cite{KS1}) call for a systematic treatment via truncation or a change in base measure.
\item[(c)] \emph{Dependent observations.} A weighted counterpart of the G{\"a}rtner--Ellis theorem for stationary ergodic sequences, along the lines of the weighted extensions in \cite{KS2}.
\item[(d)] \emph{Composite hypotheses.} A weighted analogue of the generalised likelihood-ratio test and its exponent, with sup/inf characterisations in terms of $D_C^{\rw}$ over the composite parameter sets.
\item[(e)] \emph{Information geometry of weighted manifolds.} Extending the Chentsov--Amari framework \cite{Chentsov1982,AmariNagaoka,Amari2016,NielsenSPL2013,N} to weighted statistical manifolds; the Fisher metric and the dually flat structure depend on symmetries that $\varphi$ breaks in a controlled way.
\end{itemize}}


\section*{Author Contributions}
Conceptualization, methodology, and supervision, M.K.; formal analysis and validation, M.K. and E.Y.K.; software, visualization, data curation, and project administration, E.Y.K.; writing--original draft preparation, M.K. and E.Y.K.; writing--review and editing, E.Y.K.; funding acquisition, M.K. and E.Y.K. All authors have read and agreed to the published version of the manuscript.

\section*{Funding}
The work by MK was carried out in the framework of a research project HSE-BR-2025-039 implemented as part of the Basic Research Program at HSE University.
The second author (E.Yu.~Kalimulina) was supported by the Ministry of Science and Higher Education of the Russian Federation under the state assignment (project FFNU-2025-0029).

\section*{Data Availability Statement}
The Python/Jupyter notebook reproducing Table 1 and Figures 1--3, together with the direct numerical-integration verification of the Gaussian formula in Section 4.1, is openly available on Zenodo.
No new experimental datasets were generated.

\section*{Conflicts of Interest}
The authors declare no conflicts of interest. The funders had no role in the design of the study; in the collection, analyses, or interpretation of data; in the writing of the manuscript; or in the decision to publish the results.


\appendix
\section{Cauchy Location--Scale Family}\label{app:cauchy}

We consider the univariate Cauchy location--scale family
\[
p_{l,s}(x)=\frac{s}{\pi\bigl(s^2+(x-l)^2\bigr)},\qquad x\in\bR,\quad l\in\bR,\ s>0.
\]
In this appendix, we set $\varphi\equiv 1$, so the weighted quantities from Section~\ref{sec:setup}
coincide with their classical counterparts (e.g.,\ $\rho^{\rw}_\alpha=\rho_\alpha$, $D^{\rw}_{B,\alpha}=D_{B,\alpha}$, $D_C^{\rw}=D_C$).
This example provides a closed-form illustration outside the exponential-family setting.

Although the main body of the paper develops \emph{weighted} Chernoff/Bhattacharyya quantities, we include the Cauchy location--scale family as an \emph{unweighted baseline} closed-form benchmark outside the exponential-family setting (note the appearance of complete elliptic integrals even in the classical case).
This appendix is not used in the proofs of the weighted results; it serves as a sanity check and illustrates the analytic complexity of $\rho_\alpha$ beyond exponential families.
Nontrivial weights must satisfy the finiteness conditions from Section~\ref{sec:setup}; for heavy-tailed Cauchy laws, common exponential tilts $\varphi(x)=e^{\gamma x}$ violate these conditions for any $\gamma\neq 0$, and in general, a nonconstant weight also breaks the symmetry leading to $\alpha^*=1/2$.

\subsection{KL Divergence (Closed Form)}

\begin{Proposition}[KL divergence between two Cauchy laws]\label{prop:cauchy-kl}
Let $\bP=\mathrm{Cauchy}(l_1,s_1)$ and $\bQ=\mathrm{Cauchy}(l_2,s_2)$ with $s_1,s_2>0$.
Then
\begin{equation}\label{eq:cauchy-kl}
D_{\rm KL}(\bP||\bQ)
=\int_{\bR} p_{l_1,s_1}(x)\,\ln\frac{p_{l_1,s_1}(x)}{p_{l_2,s_2}(x)}\,\rd x
=\ln\frac{(s_1+s_2)^2+(l_1-l_2)^2}{4s_1s_2}.
\end{equation}
\end{Proposition}

\subsection{Chernoff Parameter and Bhattacharyya Coefficient}

Recall the (unweighted) $\alpha$-Bhattacharyya coefficient and distance
\[
\rho_\alpha(\bP,\bQ)=\int_{\bR} p_{l_1,s_1}(x)^\alpha\,p_{l_2,s_2}(x)^{1-\alpha}\,\rd x,
~
D_{B,\alpha}(\bP,\bQ)=-\ln \rho_\alpha(\bP,\bQ),~ \alpha\in[0,1].
\]

\begin{Proposition}[Chernoff parameter for Cauchy]\label{prop:cauchy-alpha-star}
For the Cauchy location--scale family (with $\varphi\equiv 1$), one has the symmetry
\begin{equation}\label{eq:cauchy-symmetry}
\rho_\alpha(\bP,\bQ)=\rho_{1-\alpha}(\bP,\bQ),\qquad \alpha\in[0,1];
\end{equation}
equivalently, $D_{B,\alpha}(\bP,\bQ)=D_{B,1-\alpha}(\bP,\bQ)$.
Consequently, if $\bP\neq \bQ$, then the Chernoff maximiser is unique and equals
\[
\alpha^*=\frac12,
\qquad\text{and hence}\qquad
D_C(\bP,\bQ)=\max_{\alpha\in[0,1]}D_{B,\alpha}(\bP,\bQ)=D_{B,1/2}(\bP,\bQ).
\]
If $\bP=\bQ$, then $D_{B,\alpha}(\bP,\bQ)\equiv 0$ and every $\alpha\in[0,1]$ is a maximiser.
\end{Proposition}

\begin{proof}
Set $F(\alpha):=\ln\rho_\alpha(\bP,\bQ)$. By H\"older's inequality, $F$ is convex on $[0,1]$,
and it is strictly convex unless $p_{l_1,s_1}/p_{l_2,s_2}$ is $\mu$-a.e.\ constant (equivalently, unless $\bP=\bQ$).
Hence $D_{B,\alpha}(\bP,\bQ)=-F(\alpha)$ is concave (strictly concave when $\bP\neq \bQ$).
For Cauchy laws, the symmetry \eqref{eq:cauchy-symmetry} holds; see, e.g., \cite{FO} for an invariance-based proof.
Therefore, $D_{B,\alpha}$ is symmetric about $1/2$, and concavity implies that it attains its maximum at $\alpha=1/2$.
Strict concavity yields uniqueness when $\bP\neq \bQ$.
\end{proof}

\subsection{Closed Form for $\rho_{1/2}$ and $D_C$}

\begin{Lemma}[A standard elliptic-integral identity]\label{lem:elliptic-identity}
Let $a,b>0$ and $d\in\bR$. Define the complete elliptic integral of the first kind
\[
K(m):=\int_{0}^{\pi/2}\frac{\rd u}{\sqrt{1-m\sin^2 u}},\qquad m\in[0,1).
\]
Then
\begin{equation}\label{eq:elliptic-identity}
\int_{-\infty}^{\infty}\frac{\rd x}{\sqrt{(x^2+a^2)\bigl((x-d)^2+b^2\bigr)}}
=
\frac{4}{\sqrt{(a+b)^2+d^2}}\,
K\!\left(\frac{(a-b)^2+d^2}{(a+b)^2+d^2}\right).
\end{equation}
\end{Lemma}

\begin{Proposition}[Bhattacharyya coefficient for Cauchy, closed form]\label{prop:cauchy-bhat-closed-form}
Let $\bP=\mathrm{Cauchy}(l_1,s_1)$ and $\bQ=\mathrm{Cauchy}(l_2,s_2)$ with $s_1,s_2>0$, and set $\delta:=l_1-l_2$.
Then
\begin{equation}\label{eq:cauchy-rho-half}
\rho_{1/2}(\bP,\bQ)
=\int_{\bR}\sqrt{p_{l_1,s_1}(x)\,p_{l_2,s_2}(x)}\,\rd x
=
\frac{4\sqrt{s_1s_2}}{\pi\sqrt{(s_1+s_2)^2+\delta^2}}\,
K\!\left(\frac{(s_1-s_2)^2+\delta^2}{(s_1+s_2)^2+\delta^2}\right).
\end{equation}
Consequently,
\begin{equation}\label{eq:cauchy-chernoff-closed-form}
D_C(\bP,\bQ)=D_{B,1/2}(\bP,\bQ)=-\ln\rho_{1/2}(\bP,\bQ),
\end{equation}
with $\rho_{1/2}(\bP,\bQ)$ given by \eqref{eq:cauchy-rho-half}.
\end{Proposition}

\begin{proof}
By definition,
\[
\sqrt{p_{l_1,s_1}(x)\,p_{l_2,s_2}(x)}
=\frac{\sqrt{s_1s_2}}{\pi\sqrt{\bigl((x-l_1)^2+s_1^2\bigr)\bigl((x-l_2)^2+s_2^2\bigr)}}.
\]
Shift $x\mapsto x+l_2$ to obtain
\[
\rho_{1/2}(\bP,\bQ)
=\frac{\sqrt{s_1s_2}}{\pi}\int_{-\infty}^{\infty}
\frac{\rd x}{\sqrt{(x^2+s_2^2)\bigl((x-\delta)^2+s_1^2\bigr)}}.
\]
Apply Lemma~\ref{lem:elliptic-identity} with $a=s_2$, $b=s_1$, $d=\delta$ to get \eqref{eq:cauchy-rho-half}.
Finally, Proposition~\ref{prop:cauchy-alpha-star} yields $D_C(\bP,\bQ)=D_{B,1/2}(\bP,\bQ)$, hence \eqref{eq:cauchy-chernoff-closed-form}.
\end{proof}

\begin{remark}[Useful special cases]\label{rem:cauchy-special-cases}
If $l_1=l_2$ (common location), then $\delta=0$ and \eqref{eq:cauchy-rho-half} reduces to
\[
\rho_{1/2}(\bP,\bQ)
=\frac{4\sqrt{s_1s_2}}{\pi(s_1+s_2)}\,
K\!\left(\left(\frac{s_1-s_2}{s_1+s_2}\right)^2\right).
\]
If $s_1=s_2=s$ (common scale), then
\[
\rho_{1/2}(\bP,\bQ)
=\frac{4s}{\pi\sqrt{4s^2+\delta^2}}\,
K\!\left(\frac{\delta^2}{4s^2+\delta^2}\right).
\]
\end{remark}

\begin{remark}[Weighted case]\label{rem:cauchy-weighted}
For a general weight $\varphi$, one has 
$$\rho^{\rw}_{1/2}(\bP,\bQ)=\int_{\bR}\varphi(x)\sqrt{p_{l_1,s_1}(x)p_{l_2,s_2}(x)}\,\rd x.$$
The symmetry \eqref{eq:cauchy-symmetry} (and hence $\alpha^*=1/2$) typically fails unless $\varphi$ is compatible
with the invariance argument used in \cite{FO}. We therefore use the Cauchy model mainly as a closed-form baseline example at $\varphi\equiv 1$.
\end{remark}

\end{document}